\documentclass[10pt, a4paper, oneside]{amsart}

\makeatletter
\def\@settitle{\begin{center}%
		\baselineskip14\p@\relax
		\normalfont\LARGE\bfseries
		\@title
	\end{center}%
}

\def\section{\@startsection{section}{1}%
	\z@{.7\linespacing\@plus\linespacing}{.5\linespacing}%
	{\normalfont\large\bfseries}}

\def\subsection{\@startsection{subsection}{2}%
	\z@{.5\linespacing\@plus.7\linespacing}{.5\linespacing}%
	{\normalfont\bfseries}}

\def\@setauthors{%
  \begingroup
  \def\thanks{\protect\thanks@warning}%
  \trivlist
  \centering\footnotesize \@topsep30\p@\relax
  \advance\@topsep by -\baselineskip
  \item\relax
  \author@andify\authors
  \def\\{\protect\linebreak}%
  \authors%
  \ifx\@empty\contribs
  \else
    ,\penalty-3 \space \@setcontribs
    \@closetoccontribs
  \fi
  \endtrivlist
  \endgroup
}

\makeatother

\usepackage[usenames, dvipsnames]{xcolor}
\definecolor{darkblue}{rgb}{0.0, 0.0, 0.45}
\definecolor{darkgreen}{rgb}{0.0, 0.45, 0}
\usepackage[colorlinks	= true,
raiselinks	= true,
linkcolor	= darkblue, 
citecolor	= Mahogany,
urlcolor	= darkgreen,
pdfauthor	= {},
pdftitle	= {},
pdfkeywords	= {},
pdfsubject	= {},
plainpages	= false]{hyperref}

\allowdisplaybreaks
\pdfoutput=1
\date{\today}

\usepackage{dsfont,amsfonts,amssymb,amsmath,amsthm}
\usepackage{mathtools, mathrsfs}

\usepackage{graphicx}
\usepackage[font=small,labelfont=rm]{subcaption}
\usepackage[font=small,margin=10pt]{caption}
\usepackage{wrapfig}

\usepackage{booktabs}
\usepackage{nicefrac}

\usepackage{algorithm,algorithmic}

\usepackage{enumitem}
\usepackage{tikz}

\theoremstyle{plain}
\newtheorem{Thm}{Theorem}[section]

\newtheorem{Prop}[Thm]{Proposition}

\newtheorem{Lem}[Thm]{Lemma}
\newtheorem{Cor}[Thm]{Corollary}
\newtheorem{As}[Thm]{Assumption}

\newtheorem{Rem}[Thm]{Remark}

\newcommand{\R}{\mathbb{R}}
\newcommand{\Ru}{\overline {\R}}

\newcommand{\ra}{\rightarrow}

\newcommand{\ol}[1]{\overline{#1}}

\newcommand{\wt}{\widetilde}
\newcommand{\wh}{\widehat}
\newcommand{\Let}{\coloneqq}

\DeclareMathOperator*{\argmin}{\arg\!\min}
\DeclareMathOperator*{\argmax}{\arg\!\max}

\newcommand{\tr}{^{\top}}

\newcommand{\norm}[1]{\left\Vert #1 \right\Vert}
\newcommand{\inner}[2]{\left\langle #1, #2 \right\rangle }

\newcommand{\opt}{_\star}
\newcommand{\PP}{\mathds{P}}
\newcommand{\EE}{\mathds{E}}

\newcommand{\setc}[1]{\mathbb{#1}}
\newcommand{\setd}[1]{\mathbb{#1}^{\mathrm{d}}}
\newcommand{\setg}[1]{\mathbb{#1}^{\mathrm{g}}}
\newcommand{\setsg}[1]{\mathbb{#1}^{\mathrm{g}}_{\mathrm{sub}}}

\newcommand{\diam}[1]{\Delta_{#1}}

\newcommand{\lftc}[1]{#1^{*}}

\newcommand{\lftd}[1]{#1^{\mathrm{d}*}}
\newcommand{\lftdd}[1]{#1^{\mathrm{d}*\mathrm{d}}}
\newcommand{\bcc}[1]{#1^{**}}

\newcommand{\bcdd}[1]{#1^{\mathrm{d}*\mathrm{d}*}}
\newcommand{\lerp}[1]{\wt{#1}}
\newcommand{\llerp}[1]{\ol{#1}}

\newcommand{\disc}[1]{#1^{\mathrm{d}}}

\newcommand{\dynw}{g}
\newcommand{\dyn}{f}
\newcommand{\dynx}{f_\mathrm{s}}

\newcommand{\cost}{C}
\newcommand{\costx}{C_\mathrm{s}}
\newcommand{\costu}{C_\mathrm{i}}

\newcommand{\dpo}{\mathcal{T}}
\newcommand{\ddpo}{\mathcal{T}^{\mathrm{d}}}

\newcommand{\cdpo}{\widehat{\mathcal{T}}}
\newcommand{\dcdpo}{\widehat{\mathcal{T}}^{\mathrm{d}}}

\DeclareMathOperator{\dist}{d}

\DeclareMathOperator{\dish}{d_H}
\DeclareMathOperator{\ord}{\mathcal{O}}
\DeclareMathOperator{\co}{co}
\DeclareMathOperator{\lip}{L}
\DeclareMathOperator{\dom}{dom}
\DeclareMathOperator{\rng}{rng}

\usepackage[square,numbers]{natbib}
\usepackage{courier} 

\graphicspath{{Fig/}}

\title[]{Fast Approximate Dynamic Programming \\ for Infinite-Horizon Markov Decision Processes}

\author[]{M.~A.~S.~Kolarijani, G.~F.~Max, and P.~Mohajerin Esfahani}
\thanks{The authors are with Delft Center for Systems and Control, Delft University of Technology, 
Delft, The Netherlands. Email: \texttt{\{M.A.SharifiKolarijani, G.F.Max, P.MohajerinEsfahani\}@tudelft.nl}.}
\thanks{This research is part of a project that has received funding from the
European Research Council (ERC) under the grant TRUST-949796.} 
\thanks{The authors are grateful to anonymous reviewers of the 35th Conference on Neural Information Processing Systems (NeurIPS 2021) for their comments concerning the three remarks in Section~\ref{sec:conclusion}.} 
\thanks{This arXiv article is the extended and slightly modified version of the article appeared in the 35th Conference on Neural Information Processing Systems (NeurIPS 2021) under the same title.}

\begin{document}

\begin{abstract}
In this study, we consider the infinite-horizon, discounted cost, optimal control of stochastic nonlinear systems with separable cost and constraints in the state and input variables. 
Using the linear-time Legendre transform, we propose a novel numerical scheme for the implementation of the corresponding value iteration (VI) algorithm in the conjugate domain. 
Detailed analyses of the convergence, time complexity, and error of the proposed algorithm are provided. 
In particular, with a discretization of size $X$ and $U$ for the state and input spaces, respectively, the proposed approach reduces the time complexity of each iteration in the VI algorithm from $\ord(XU)$ to $\ord (X+U)$, by replacing the minimization operation in the primal domain with a simple addition in the conjugate domain. 


\end{abstract}

\maketitle

\section{Introduction}
\label{sec:intro}

Value iteration (VI) is one of the most basic and widespread algorithms employed for tackling problems in reinforcement learning (RL) and optimal control~\cite{Bertsekas19,Sutton18} formulated as Markov decision processes (MDPs). 
The VI algorithm simply involves the consecutive applications of the dynamic programming (DP) operator 
\begin{equation*} \label{eq:BE}
\dpo J (x_t)  = \min_{u_t} \big\{ \cost(x_t,u_t) + \gamma \EE J(x_{t+1}) \big\},
\end{equation*}
where $\cost(x_t,u_t)$ is the cost of taking the control action~$u_t$ at the state~$x_t$. 
This fixed point iteration is known to converge to the optimal value function for discount factors $\gamma \in (0,1)$.  
However, this algorithm suffers from a high computational cost for large-scale finite state spaces. 
For problems with a continuous state space, the DP operation becomes an infinite-dimensional optimization problem, rendering the exact implementation of VI impossible in most cases. 
A common solution is to incorporate function approximation techniques and compute the output of the DP operator for a finite sample (i.e., a discretization) of the underlying continuous state space. 
This approximation again suffers from a high computational cost for fine discretizations of the state space, particularly in high-dimensional problems. We refer the reader to~\cite{Bertsekas19, Pow11} for various approximation schemes for the implementation of VI. 
 
For some problems, however, it is possible to partially address this issue by using duality theory, i.e., approaching the minimization problem in the conjugate domain. 
In particular, as we will see in Section~\ref{sec:conj DP}, the minimization in the primal domain in DP can be transformed to a simple addition in the dual domain, at the expense of three conjugate transforms. 
However, the proper application of this transformation relies on efficient numerical algorithms for conjugation. 
Fortunately, such an algorithm, known as linear-time Legendre transform (LLT), has been developed in the late 90s~\cite{Lucet97}. 
Other than the classical application of LLT (and other fast algorithms for conjugate transform) in solving the Hamilton-Jacobi equation \cite{Achdou14, Corrias96, Costes14}, these algorithms are used in image processing \cite{Lucet09}, thermodynamics \cite{Contento15}, and optimal transport \cite{Jacobs19}.

The application of conjugate duality for the DP problem is not new and goes back to Bellman~\cite{Bell62}. 
Further applications of this idea for reducing the computational complexity were later explored in~\cite{Esog90,Klein91}. 
However, surprisingly, the application of LLT for solving discrete-time optimal control problems has been limited. 
In particular, in \cite{Caprio16}, the authors propose the ``fast value iteration'' algorithm 
(without a rigorous analysis of the complexity and error of the proposed algorithm) 
for a particular class of infinite-horizon optimal control problems with state-independent stage cost $\cost (x,u) = \cost (u)$ and deterministic linear dynamics $x_{t+1} = Ax_t + B u_t$, where $A$ is a non-negative, monotone, invertible matrix. 
More recently, in \cite{Kolari21dCDParxiv}, we also considered the application of LLT for solving the DP operation in finite-horizon, optimal control of input-affine dynamics $x_{t+1} = \dynx (x_t) + B u_t$ with separable cost $\cost (x,u) = \costx (x)+  \costu (u)$. 
In particular, we introduced the ``discrete conjugate DP'' (d-CDP) operator, and provided a detailed analysis of its complexity and error. 
As we will discuss shortly, the current study is an extension of the corresponding d-CDP algorithm that, among other things, considers infinite horizon, discounted cost problems. 
We note that the algorithms developed in~\cite{Felzen12,Lucet09} for ``distance transform'' can also potentially tackle the optimal control problems similar to the ones of interest in the current study. 
In particular, these algorithms require the stage cost to be reformulated as a convex function of the ``distance'' between the current and next states. 
While this property might arise naturally, it
can generally be restrictive, as it is in the problem class considered in this study. 
Another line of work that is closely related to ours involves utilizing max-plus algebra in solving deterministic, continuous-state, continuous-time, optimal control problems; see, e.g.,~\cite{Akian08, McEn03}. 
These works exploit the compatibility of the DP operation with max-plus operations, 
and approximate the value function as a max-plus linear combination. 
Recently, in~\cite{Bach19,Bach20}, the authors used this idea to propose an approximate VI algorithm for continuous-state, deterministic MDPs. 
In this regard, we note that the proposed approach in the current study also involves approximating the value function as a max-plus linear combination, namely, the maximum of affine functions. 
The key difference is however that by choosing a grid-like (factorized) set of slopes for the linear terms (i.e., the basis of the max-plus linear combination), we take advantage of the linear time complexity of LLT in computing the constant terms (i.e., the coefficients of the max-plus linear combination).

\noindent\textbf{Main contribution.} 
In this study, we focus on an approximate implementation of VI involving discretization of the state and input spaces for solving the optimal control problem of discrete-time systems, with continuous state-input space. 
Building upon our earlier work~\cite{Kolari21dCDParxiv}, we employ conjugate duality to speed up VI for problems with separable stage cost (in state and input) and input-affine dynamics. 
We propose the \emph{conjugate} VI (ConjVI) algorithm based on a modified version of the d-CDP operator introduced in~\cite{Kolari21dCDParxiv}, 
and extend the existing results in three directions: 
We consider \emph{infinite-horizon, discounted cost} problems with \emph{stochastic dynamics}, while incorporating a \emph{numerical scheme for approximation of the conjugate of input cost}. 
The main contributions of this paper are as follows:

\begin{itemize}
\item[(i)] we provide sufficient conditions for the convergence of ConjVI (Theorem~\ref{thm:convergence});
\item[(ii)] we show that ConjVI can achieve a linear time complexity of $\ord (X+U)$ in each iteration (Theorem~\ref{thm:complexity}), compared to the quadratic time complexity of $\ord(XU)$ of the standard VI, where $X$ and $U$ are the cardinalities of the discrete state and input spaces, respectively; 
\item[(iii)] we analyze the error of ConjVI (Theorem~\ref{thm:error}) and use that result to provide specific guidelines on the construction of the discrete dual domain (Section~\ref{subsec:grid construction}); 
\item[(iv)] we provide a MATLAB package for the implementation of the proposed ConjVI algorithm~\cite{Kolari21ConjVIMAT}.  
\end{itemize}

\noindent\textbf{Paper organization.} The problem statement and its standard solution via the VI algorithm (in primal domain) are presented in Section~\ref{sec:prim DP}. 
In Section~\ref{sec:conj DP}, we present our main results: We begin with presenting the class of problems that are of interest, and then introduce the alternative approach for VI in the conjugate domain and its numerical implementation. 
The theoretical results on the convergence, complexity, and error of the proposed algorithm along with the guidelines on the construction of dual grids are also provided in this section. 
In Section~\ref{sec:numerical ex}, we compare the performance of the ConjVI with that of VI algorithm through three numerical examples. 
Section~\ref{sec:conclusion} concludes the paper with some final remarks.
All the technical proofs are provided in Appendix~\ref{sec:proofs}.

\noindent\textbf{Notations.} We use $\R$ and $\ol{\R} = \R \cup \{\infty\}$ to denote the real line and the extended reals, respectively, and $\EE_w [\cdot]$ to denote expectation with respect (w.r.t.) to the random variable $w$. 
The standard inner product in $\R^n$ and the corresponding induced 2-norm are denoted by $\inner{\cdot}{\cdot}$ and $\norm{\cdot}_2$, respectively. 
We also use $\norm{\cdot}_2$ to denote the operator norm (w.r.t. the 2-norm) of a matrix;
i.e., for $A \in \R^{m\times n}$, we denote $\norm{A}_2 = \sup \{ \norm{Ax}_2 : \norm{x}_2 = 1\}$.
The infinity-norm is denoted by $\norm{\cdot}_{\infty}$.

Continuous (infinite, uncountable) sets are denoted as $\setc{X}, \setc{Y}, \ldots$. 
For \emph{finite} (discrete) sets, we use the superscript~$\mathrm{d}$ as in $\setd{X}, \setd{Y}, \ldots$ to differentiate them from infinite sets. 
Moreover, we use the superscript~$\mathrm{g}$ to differentiate \emph{grid-like} finite sets. 
Precisely, a grid~$\setg{X} \subset \R^n$ is the Cartesian product $\setg{X} = \Pi_{i=1}^{n} \setg{X}_{i} = \setg{X}_{1} \times \ldots \times \setg{X}_{n}$, where $\setg{X}_{i}$ is a finite subset of $\R$. 
We also use $\setsg{X}$ to denote the \emph{sub-grid} of $\setg{X}$ derived by omitting the smallest and the largest elements of $\setg{X}$ in each dimension.
The cardinality of a finite set~$\setd{X}$ or $\setg{X}$ is denoted by $X$. 
Let $\setc{X}, \setc{Y}$ be two arbitrary sets in $\R^n$.
The convex hull of $\setc{X}$ is denoted by $\co (\setc{X})$. 
The diameter of~$\setc{X}$ is defined as $\Delta_{\setc{X}} \Let \sup_{x,\tilde{x} \in \setc{X}} \norm{x-\tilde{x}}_2$.  
We use $\dist (\setc{X},\setc{Y}) \Let \inf_{x \in \setc{X}, y \in \setc{Y}} \norm{x-y}_2$ to denote the distance between $\setc{X}$ and $\setc{Y}$. 
The one-sided Hausdorff distance \emph{from}~$\setc{X}$ \emph{to}~$\setc{Y}$ is defined as $\dish(\setc{X}, \setc{Y}) \Let \sup_{x \in \setc{X}} \inf_{y \in \setc{Y}} \norm{x-y}_2$. 

Let $h: \R^n \ra \Ru$ be an extended real-valued function with a non-empty effective domain~$\dom(h) = \setc{X} \Let \{x \in \R^n: h(x) < \infty \}$, and range $\rng (h) = \max_{x \in \setc{X}} h(x) -\min_{x \in \setc{X}} h(x)$. 
We use $\disc{h}: \setd{X} \ra \Ru$ to denote the discretization of $h$, where $\setd{X}$ is a finite subset of $\R^n$. 
Whether a function is discrete is usually also clarified by providing its domain explicitly. 
We particularly use this notation in combination with a second operation to emphasize that the second operation is applied on the discretized version of the operand. 
E.g., we use $\lerp{\disc{h}}:\R^n \ra \Ru$ to denote a \emph{generic extension} of $\disc{h}$. 
If the domain~$\setd{X} = \setg{X}$ of $\disc{h}$ is grid-like, we then use $\llerp{\disc{h}}$ (as opposed to $\lerp{\disc{h}}$) for the extension using \emph{multi-linear interpolation and extrapolation (LERP)}.
The Lipschtiz constant of~$h$ over a set~$\setc{Y} \subset \dom (h)$ is denoted by $\lip (h; \setc{Y}) \Let  \sup_{x,y \in \setc{Y}} |h(x)-h(y)| / \norm{x-y}_2$.  
We also denote $\lip (h) \Let \lip \big(h; \dom (h)\big)$ and 
$\setc{L}(h) \Let \Pi_{i=1}^{n} \left[\lip_i^-(h), \lip_i^+(h)\right]$, where 
$\lip_i^+(h)$ (resp. $\lip_i^-(h)$) is the maximum (resp. minimum) slope of the function $h$ along the $i$-th dimension, 
The subdifferential of~$h$ at a point~$x \in \setc{X}$ is defined as
$
\partial h(x) \Let \big\{y \in \R^n: h(\tilde{x}) \geq h(x) + \inner{y}{\tilde{x}-x}, \forall \tilde{x} \in \setc{X} \big\}
$. 
Note that $\partial h(x) \subseteq \setc{L}(h)$ for all $x \in \setc{X}$; in particular, $\setc{L}(h) = \cup_{x \in \setc{X}} \partial h(x)$ if $h$ is convex. 
The Legendre-Fenchel transform (convex conjugate) of~$h$ is the function~$\lftc{h}: \R^n \ra \ol{\R}$, defined by $\lftc{h}(y) = \sup_{x} \left\{ \inner{y}{x} - h(x) \right\}$. 
Note that the conjugate function~$\lftc{h}$ is convex by construction. 
We again use the notation $\lftd{h}$ to emphasize the fact that the domain of the underlying function is \emph{finite}, 
that is, $\lftd{h}(y) = \sup_{x \in \setd{X}} \left\{ \inner{y}{x} - h(x) \right\}$. 
The biconjugate and discrete biconjugate operators are defined accordingly and denoted by $\bcc{[\cdot]} = \lftc{[\lftc{[\cdot]}]}$ and $\bcdd{[\cdot]} = \lftd{[\lftd{[\cdot]}]}$, respectively. 

We report the complexities using the standard big-O notations $\ord$ and $\wt{\ord}$, where the latter hides the logarithmic factors. 
In this study, we are mainly concerned with the dependence of the computational complexities on \emph{the size of the finite sets} involved (discretization of the primal and dual domains). 
In particular, we ignore the possible dependence of the computational complexities on the dimension of the variables, unless they appear in the power of the size of those discrete sets; e.g., the complexity of a single evaluation of an analytically available function is taken to be of $\ord(1)$, regardless of the dimension of its input and output arguments.

\section{VI in primal domain}
\label{sec:prim DP}

We are concerned with the infinite-horizon, discounted cost, optimal control problems of the form
\begin{align*}
J\opt(x) = &\min \ \EE_{w_t} \left[ \sum_{t=0}^{\infty} \gamma^t C(x_t,u_t) \bigg| x_0 = x \right] \\
&\ \text{s.t.}  \ \ \ x_{t+1} = \dynw(x_t,u_t,w_t),\ x_t \in \setc{X},\ u_t \in \setc{U},\ w_t \sim \PP (\setc{W}), \quad \forall t \in \{0,1,\ldots\},
\end{align*}
where $x_t \in \R^n$, $u_t \in \R^m$, and $w_t \in \R^l$ are the state, input and disturbance variables at time $t$, respectively; 
$\gamma \in (0,1)$ is the discount factor; 
$C: \setc{X}\times\setc{U} \ra \R$ is the stage cost; 
$\dynw:\R^n\times\R^m\times\R^l \ra \R^n$ describes the dynamics; 
$\setc{X}\subset \R^n$ and $\setc{U}\subset \R^m$ describe the state and input constraints, respectively;
and, $\PP (\cdot)$ is the distribution of the disturbance over the support~$\setc{W} \subset \R^l$. 
Assuming the stage cost $C$ is bounded, the optimal value function solves the Bellman equation $J\opt = \dpo J\opt$, where $\dpo$ is the DP operator ($C$ and $J$ are extended to infinity outside their effective domains)~\cite[Prop.~1.2.2]{Bertsekas07}
\begin{equation} \label{eq:DP op}
\dpo J(x) \Let \min_{u} \left\{ \cost(x,u) + \gamma \cdot \EE_w J\big(\dynw(x,u,w)\big) \right\}, \quad \forall x \in \setc{X}.
\end{equation} 
Indeed, $\dpo$ is $\gamma$-contractive in the infinity-norm, i.e., $\norm{\dpo J_1 - \dpo J_2}_{\infty} \leq \gamma \norm{J_1 - J_2}_{\infty}$ \cite[Prop.~1.2.4]{Bertsekas07}. 
This property then gives rise to the VI algorithm $J_{k+1} = \dpo J_k$ which converges to $J\opt$ as $k \ra \infty$, for arbitrary initialization $J_0$. 
Moreover, assuming that the composition $J \circ \dynw$ (for each $w$) and the cost $\cost$ are jointly convex in the state and input variables, $\dpo$ also preserves convexity \cite[Prop.~3.3.1]{Bertsekas09}. 

For the numerical implementation of VI, we need to address three issues. 
First, we need to compute the expectation in~\eqref{eq:DP op}. 
To simplify the exposition and include the computational cost of this operation explicitly, we consider disturbances with finite support in this study:

\begin{As}[Disturbance with finite support]\label{As:distr}
The disturbance $w$ has a finite support $\setd{W} \subset \R^l$ with a given probability mass function (p.m.f.) $p: \setd{W} \ra [0,1]$. 
\end{As}

Under the preceding assumption, we have $\EE_w J\big(\dynw(x,u,w)\big) = \sum_{w \in \setd{W}} p(w) \cdot J\big(\dynw(x,u,w)\big)$.\footnote{ Indeed, $\setd{W}$ can be considered as a finite approximation of the true support $\setc{W}$ of the disturbance. Moreover, one can consider other approximation schemes, such as Monte Carlo simulation, for this expectation operation.} 
The second and more important issue is that the optimization problem~\eqref{eq:DP op} is infinite-dimensional for the continuous state space $\setc{X}$. 
This renders the exact implementation of VI impossible, except for a few cases with available closed-form solutions. 
A common solution to this problem is to deploy a sample-based approach, accompanied by a function approximation scheme. 
To be precise, for a finite subset $\setd{X}$ of $\setc{X}$, at each iteration $k=0,1,\ldots$, we take the discrete function $\disc{J}_k: \setd{X} \ra \R$ as the input, 
and compute the discrete function $\disc{J}_{k+1} = \disc{\left[\dpo \lerp{\disc{J}_k}\right]}: \setd{X}\ra \R$, 
where $\lerp{\disc{J}_k}: \setc{X} \ra \R$ is an extension of $\disc{J}_k$.\footnote{ The extension can be considered as a generic parametric approximation $\wh{J}_{\theta_k}:\setc{X}\ra\R$, where the parameters $\theta_k$ are computed using regression, i.e., by fitting $\wh{J}_{\theta_k}$ to the data points $\disc{J}_k:\setd{X}\ra\R$.} 
Finally, for each $x \in \setd{X}$, we have to solve the minimization problem in~\eqref{eq:DP op} over the control input. 
Since this minimization problem is often a difficult, non-convex problem, a common approximation again involves enumeration over a discretization $\setd{U} \subset \setc{U}$ of the input space. 

Incorporating these approximations, we end up with the approximate VI algorithm $\disc{J}_{k+1} = \ddpo \disc{J}_k$, characterized by the \emph{discrete} DP (d-DP) operator
\begin{equation} \label{eq:d-DP op}
\ddpo \disc{J}(x) \Let \min_{u \in \setd{U}} \bigg\{ \cost(x,u) + \gamma \cdot \sum_{w \in \setd{W}} p(w) \cdot  \lerp{\disc{J}}\big(\dynw(x,u,w)\big) \bigg\}, \quad \forall x \in \setd{X}.
\end{equation}
The convergence of approximate VI described above depends on the properties of the extension operation $\lerp{[\cdot]}$. 
In particular, if $\lerp{[\cdot]}$ is non-expansive (in the infinity-norm), then $\ddpo$ is also $\gamma$-contractive. 
For example, for a grid-like discretization of the state space $\setd{X} = \setg{X}$, the extension using \emph{interpolative} LERP is non-expansive; see Lemma~\ref{lem:lerp non-exp}. 
The error of this approximation ($\lim \norm{\disc{J}_{k} - \disc{J}\opt}_{\infty}$) also depends on the extension operation $\lerp{[\cdot]}$ and its representative power. 
We refer the interested reader to \cite{Bertsekas07,Busoniu17,Pow11} for detailed discussions on the convergence and error of different approximation schemes for VI.   

The d-DP operator and the corresponding approximate VI algorithm will be our benchmark for evaluating the performance of the alternative algorithm developed in this study. 
To this end, we finish this section with some remarks on the time complexity of the d-DP operation. 
Let the time complexity of a single evaluation of the extension operator $\lerp{[\cdot]}$ in \eqref{eq:d-DP op} be of $\ord (E)$.\footnote{ For example, for the linear approximation $\lerp{\disc{J}}(x) = \sum_{i=1}^{B} \alpha_i \cdot b_i(x)$, we have $E = B$ (the size of the basis), while for the kernel-based approximation $\lerp{\disc{J}} (x)= \sum_{\bar{x} \in \setd{X}}\alpha_{\bar{x}} \cdot r(x,\bar{x})$, we generally have $E \leq X$. 
In particular, if $\setd{X} = \setg{X}$ is grid-like, and $\lerp{\disc{J}} =  \llerp{\disc{J}}$ is approximated using LERP, then $E = \log X$ \cite[Rem.~2.2]{Kolari21dCDParxiv}.} 
Then, the time complexity of the d-DP operation~\eqref{eq:d-DP op} is of $\ord\big(XUWE\big)$. 
In this regard, note that the scheme described above essentially involves approximating a continuous-state/action MDP with a finite-state/action MDP, and then applying VI. 
This, in turn, implies the lower bound $\Omega(XU)$ for the time complexity (corresponding to enumeration over $u \in \setd{U}$ for each $x \in \setd{X}$). 
This lower bound is also compatible with the best existing time complexities in the literature for VI for finite MDPs; see, e.g., \cite{Bach19, Sidford18}. 
However, as we will see in the next section, for a particular class of problems, it is possible to exploit the structure of the underlying continuous system to achieve a better time complexity in the corresponding discretized problem.

\section{Reducing complexity via conjugate duality}
\label{sec:conj DP}

In this section, we present the class of problems that allows us to employ conjugate duality 
and propose an alternative path for solving the corresponding DP operator. 
We also present the numerical scheme for implementing the proposed alternative path and analyze its convergence, complexity, and error. 
We note that the proposed algorithm and its analysis are based on the d-CDP algorithm presented in \cite[Sec.~5]{Kolari21dCDParxiv} for finite-horizon, optimal control of deterministic systems. 
Here, we extend those results for infinite-horizon, discounted cost, optimal control of stochastic systems. 
Moreover, unlike \cite{Kolari21dCDParxiv}, our analysis includes the case where the conjugate of input cost is not analytically available and has to be computed numerically; see \cite[Assump.~5.1]{Kolari21dCDParxiv} for more details. 


\subsection{VI in conjugate domain} \label{subsec:CDP}

Throughout this section, we assume that the problem data satisfy the following conditions. 

\begin{As}[Problem class]\label{As:prob data} 
The problem data has the following properties:
\begin{enumerate}[label=(\roman*)]

\item \label{As:dyn} The \emph{dynamics} is of the form 
$
\dynw(x,u,w) = \dyn (x,u) + w = \dynx (x) + Bu + w 
$, 
with additive disturbance, where $\dynx: \R^n \ra \R^n$ is a Lipschitz continuous, possibly nonlinear map, and $B \in \R^{n\times m}$. 

\item \label{As:cost} The \emph{stage cost} $\cost$ is separable in state and input; 
that is,~$\cost(x,u) = \costx(x)+\costu(u)$, where the state cost $\costx: \setc{X} \ra \R$ and the input cost $\costu: \setc{U} \ra \R$ are Lipschitz continuous. 

\item \label{As:constr} The \emph{constraint} sets $\setc{X} \subset \R^n$ and $\setc{U} \subset \R^m$ are compact. 
Moreover, for each $x \in \setc{X}$, the set of admissible inputs 
$
\setc{U}(x) \Let \{ u \in \setc{U}:  \dynw(x,u,w) \in \setc{X}, \ \forall w \in \setd{W} \}
$ 
is nonempty. 

\end{enumerate}
\end{As}

Some remarks are in order regarding the preceding assumptions. 
We first note that the setting of Assumption~\ref{As:prob data} goes beyond the classical LQR. 
In particular, it includes nonlinear dynamics, state and input constraints, and non-quadratic stage costs.
Second, the properties laid out in Assumption~\ref{As:prob data} imply that the set of admissible inputs $\setc{U}(x)$ is a compact set for each $x \in \setc{X}$. 
This, in turn, implies that the optimal value in~\eqref{eq:DP op} is achieved if $J: \setc{X} \ra \R$ is also assumed to be lower semi-continuous. 
Finally, as we discuss shortly, the two assumptions on the dynamics and the cost play an essential role in the derivation of the alternative algorithm and its computationally efficient implementation.  

 

For the problem class of Assumption~\eqref{As:prob data}, we can use duality theory to present an alternative path for computing the output of the DP operator. 
This path forms the basis for the algorithm proposed in this study. 
To this end, let us fix $x \in \setc{X}$ and consider the following reformulation of the optimization problem~\eqref{eq:DP op}
\begin{align*}
\dpo J(x) = &\costx(x) + \min_{u, z} \left\{ \costu (u) + \gamma \cdot \EE_w J (z+w): z = \dyn(x,u)  \right\},
\end{align*}
where we used additivity of disturbance and separability of stage cost. 
The corresponding dual problem then reads as
\begin{align} \label{eq:CDP op}
&\cdpo J(x) \Let \costx(x) + \max_{y} \ \min_{u,z} \left\{ \costu (u) + \gamma \cdot \EE_w J (z+w) + \inner{y}{\dyn(x,u)-z}  \right\}, 
\end{align}
where $y \in \R^n$ is the dual variable corresponding to the equality constraint. 
For the dynamics of Assumption~\ref{As:prob data}-\ref{As:dyn}, we can then obtain the following representation for the dual problem.

\begin{Prop}[CDP operator]\label{prop:CDP op conj}
The dual problem~\eqref{eq:CDP op} equivalently reads as
\begin{subequations}\label{eq:CDP op conj}
\begin{align}
& \epsilon(x) \Let \gamma \cdot\EE_w J (x+w), &  x \in \setc{X}, \label{eq:CDP op eps}\\
& \phi(y) \Let \lftc{\costu}(-B\tr y) + \lftc{\epsilon}(y), & y \in  \R^n, \label{eq:CDP op phi} \\
& \cdpo J(x) = \costx(x) + \lftc{\phi} \big( \dynx(x) \big), &  x \in \setc{X}, \label{eq:CDP op TV} 
\end{align}
\end{subequations}
where $\lftc{[\cdot]}$ denotes the conjugate operation.
\end{Prop}

Following~\cite{Kolari21dCDParxiv}, we call the operator $\cdpo$ in \eqref{eq:CDP op conj} the \emph{conjugate} DP (CDP) operator. 
We next provide an alternative representation of the CDP operator that captures the essence of this operation.

\begin{Prop}[CDP reformulation] \label{prop:CDP reform}
The CDP operator~$\cdpo$ equivalently reads as
\begin{equation} \label{eq:CDP op alt}
\cdpo J(x) = \costx(x) + \min_{u} \left\{ \bcc{\costu}(u) + \gamma \cdot \bcc{[\EE_w J (\cdot+w)]} \big( \dyn(x,u) \big) \right\},
\end{equation}
where $\bcc{[\cdot]}$ denotes the biconjugate operation.
\end{Prop} 

The preceding result implies that the indirect path through the conjugate domain essentially involves substituting the input cost and (expectation of the) value function by their biconjugates. 
In particular, it points to a sufficient condition for zero duality gap. 

\begin{Cor}[Equivalence of $\dpo$ and $\cdpo$] \label{cor:equivalnece CDP and DP}
If $\costu : \setc{U} \ra \R$ and $J : \setc{X} \ra \R$ are convex, then $\cdpo J = \dpo J$.
\end{Cor} 

Hence, $\cdpo$ has the same properties as $\dpo$ if $\costu$ and $J$ are convex. 
More importantly, if $\dpo$ and $\cdpo$ preserve convexity, then the \emph{conjugate} VI (ConjVI) algorithm $J_{k+1} = \cdpo J_k$ also converges to the optimal value function $J\opt$, with arbitrary convex initialization $J_0$. 
For convexity to be preserved, however, we need more assumptions: 
First, the state cost $\costx : \setc{X} \ra \R$ needs to be also convex. 
Then, for $\cdpo J$ to be convex, a sufficient condition is convexity of $J \circ \dyn$ (jointly in $x$ and $u$), given that $J$ is convex. 
The following assumption summarizes the sufficient conditions for equivalence of VI and ConjVI algorithms.

\begin{As}[Convexity]\label{As:convex} 
Consider the following properties for the constraints, costs, and dynamics:
\begin{enumerate}[label=(\roman*)]

\item  The sets $\setc{X} \subset \R^n$ and $\setc{U} \subset \R^m$ are convex.

\item  The costs $\costx: \setc{X} \ra \R$ and $\costu: \setc{U} \ra \R$ are convex.

\item The deterministic dynamics $\dyn:\R^n\times\R^m\ra\R^n$ is such that given a convex function~$J:\setc{X} \ra R$, the composition~$J \circ \dyn$ is jointly convex in the state and input variables. 

\end{enumerate}
\end{As}

We note that the last condition in the preceding assumption usually does not hold for nonlinear dynamics, however, for $\dynx (x) = Ax$ with $A \in \R^{n\times n}$, this is indeed the case for problems satisfying Assumptions~\ref{As:prob data} and \ref{As:convex} ~\cite{Bertsekas73}.
Note that, if convexity is not preserved, then the alternative path suffers from duality gap in the sense that in each iteration it uses the \emph{convex envelope} of (the expectation of) the output of the previous iteration.

\subsection{ConjVI algorithm} \label{subsec:alg} 
The approximate ConjVI algorithm involves consecutive applications of an approximate implementation of the CDP operator~\eqref{eq:CDP op conj} until some termination condition is satisfied. 
Algorithm~\ref{alg:d-CDP separ cost} provides the pseudo-code of this procedure. 
In particular, we consider solving~\eqref{eq:CDP op conj} for a finite set $\setd{X} \subset \setc{X}$, and terminate the iterations when the difference between two consecutive discrete value functions (in the infinity-norm) is less than a given constant $e_\mathrm{t} >0$; see Algorithm~\ref{alg:d-CDP separ cost}:\ref{line_alg2:terminate}. 
Since we are working with a finite subset of the state space, we can restrict the feasibility condition of Assumption~\ref{As:prob data}-\ref{As:constr} to all $x\in \setd{X}$ (as opposed to all $x\in \setc{X}$):

\begin{As}[Feasibile discretization]\label{As:feasible discrete}
The set of admissible inputs $\setc{U}(x)$ is nonempty for all $x \in \setd{X}$.
\end{As} 

\begin{algorithm}[t]
   \caption{ConjVI: Approximate VI in conjugate domain}
   \label{alg:d-CDP separ cost}
\begin{algorithmic}[1]
\begin{small}

	\REQUIRE dynamics $\dynx: \R^n \ra \R^n, \ B \in \R^{n\times m}$; 
	finite state space~$\setd{X} \subset \setc{X}$;
	finite input space~$\setd{U} \subset \setc{U}$; 
	state cost function~$\disc{\costx}: \setd{X} \ra \R$;
	input cost function~$\disc{\costu}: \setd{U} \ra \R$;
	finite disturbance space~$\setd{W}$ and its p.m.f.~$p: \setd{W} \ra [0,1]$;
	discount factor~$\gamma$; 
	termination bound~$e_\mathrm{t}$.
	\ENSURE discrete value function $\disc{\wh{J}}: \setd{X} \ra \R$.
  	
  	
  	\textit{initialization:}
  	\STATE construct the grid $\setg{V}$; \label{line_alg2:const V}
	\STATE use LLT to compute $\lftdd{\costu}: \setg{V} \ra \R$ from $\disc{\costu}: \setd{U} \ra \R$; \label{line_alg2:LLT of Ci}
	\STATE construct the grid $\setg{Z}$; \label{line_alg2:const Z}
	\STATE construct the grid $\setg{Y}$; \label{line_alg2:const Y}  	
  	\STATE $\disc{J}(x) \gets 0$ for $x \in \setd{X}$; \label{line_alg2:init 1}
  	\STATE $\disc{J}_+(x) \gets \disc{\costx}(x) - \min \disc{\costu}$ for $x \in \setd{X}$; \label{line_alg2:init 2}

  	\textit{iteration:}
  	
  	\WHILE{$\norm{\disc{J}_{+} - \disc{J}}_{\infty} \ge e_\mathrm{t}$} \label{line_alg2:terminate}
  	
  		\STATE $ \disc{J} \gets \disc{J}_{+} $; \label{line_alg2:iter 1}
  	
  	\textit{d-CDP operation:}
  		\STATE $\disc{\varepsilon}(x) \gets \gamma \cdot \sum_{w \in \setd{W}} p(w) \cdot \lerp{\disc{J}} (x+w)$ for $x \in \setd{X}$; \label{line_alg2:LERP of V}
    	\STATE use LLT to compute $\lftdd{\varepsilon}: \setg{Y} \ra \Ru$ from $\disc{\varepsilon}: \setd{X} \ra \R$; \label{line_alg2:LLT of V}
    
	    \FOR{each $y \in \setg{Y}$}
    
 			\STATE  use LERP to compute $\llerp{\lftdd{\costu}}(-B^{\top} y)$ from $\lftdd{\costu}: \setg{V} \ra \R$; \label{line_alg2:LERP of Ci}
     	
    		\STATE $\disc{\varphi}(y) \gets \llerp{\lftdd{\costu}}(-B^{\top} y) + \lftdd{\varepsilon}(y)$; \label{line_alg2:phi}
    	
   		\ENDFOR

    
    	\STATE use LLT to compute $\lftdd{\varphi}: \setg{Z} \ra \R$ from $\disc{\varphi}: \setg{Y} \ra \R $; \label{line_alg2:LLT of phi}
    
    
    	\FOR{each $x \in \setd{X}$}
    
 			\STATE  use LERP to compute $\llerp{\lftdd{\varphi}}\big( \dynx(x) \big)$ from $\lftdd{\varphi}: \setg{Z} \ra \R$; \label{line_alg2:LERP of phi}
     	
    		\STATE $\disc{J}_{+}(x) \gets \costx (x) + \llerp{\lftdd{\varphi}}\big( \dynx(x) \big)$;
    	
   		\ENDFOR
   		
   	\ENDWHILE
   	
   	\STATE output $ \disc{\wh{J}} \gets \disc{J}_+ $.
 
\end{small}  	
  	
\end{algorithmic}
\end{algorithm}

In what follows, we describe the main steps within the initialization and iterations of Algorithm~\ref{alg:d-CDP separ cost}. 
In particular, the conjugate operations in~\eqref{eq:CDP op conj} are handled numerically via the linear-time Legendre transform (LLT) algorithm~\cite{Lucet97}. 
LLT is an efficient algorithm for computing the \emph{discrete} conjugate function over a finite \emph{grid-like} dual domain. 
Precisely, to compute the conjugate of the function~$h: \setc{X} \ra \R$, LLT takes its discretization $\disc{h}: \setd{X} \ra \R$ as an input, and outputs $\lftdd{h}: \setg{Y} \ra \R$, for the grid-like dual domain $\setg{Y}$. 
We refer the reader to \cite{Lucet97} for a detailed description of LLT. 
The main steps of the proposed approximate implementation of the CDP operator~\eqref{eq:CDP op conj} are as follows:

\begin{itemize}
\item[(i)] For the expectation operation in~\eqref{eq:CDP op eps}, by Assumption~\ref{As:distr}, we again have $$\EE_w J (\cdot+w) = \sum_{w \in \setd{W}} p(w) \cdot J (\cdot+w).$$ 
Hence, we need to pass the value function $\disc{J}: \setd{X} \ra \R$ through the ``scaled expection filter'' to obtain $\disc{\varepsilon}: \setd{X} \ra \Ru$ in \eqref{eq:d-CDP op eps} as an approximation of $\epsilon$ in \eqref{eq:CDP op eps}. 
Notice that here we are using an extension $\lerp{\disc{J}}:\setc{X} \ra \R$ of $\disc{J}$ (recall that we only have access to the discrete value function $\disc{J}$). 
\item[(ii)] To compute $\phi$ in~\eqref{eq:CDP op phi}, we need access to two conjugate functions: 
\begin{itemize}
\item[(a)] For $\lftc{\epsilon}$, we use the approximation $\lftdd{\varepsilon}: \setg{Y}\ra \R$ in \eqref{eq:d-CDP op eps conj}, by applying LLT to the data points $\disc{\varepsilon}: \setd{X} \ra \Ru$ for a properly chosen state dual grid $\setg{Y} \subset \R^n$. 
\item[(b)] If the conjugate $\lftc{\costu}$ of the input cost is not analytically available, we approximate it as follows: For a properly chosen input dual grid $\setg{V} \subset \R^m$, we employ LLT to compute $\lftdd{\costu}: \setg{V} \ra \R$ in~\eqref{eq:d-CDP op Ci conj}, 
using the data points $\disc{\costu}: \setd{U} \ra \R$, where $\setd{U}$ is a finite subset of $\setc{U}$. 
\end{itemize}
With these conjugate functions at hand, we can now compute $\disc{\varphi}: \setg{Y} \ra \R$ in \eqref{eq:d-CDP op phi}, as an approximation of $\phi$ in \eqref{eq:CDP op phi}. 
In particular, notice that we use the LERP extension $\llerp{\lftdd{\costu}}$ of $\lftdd{\costu}$ to approximate $\lftd{\costu}$ at the required point $(-B\tr y)$ for each $y \in \setg{Y}$.
\item[(iii)] To be able to compute the output according to~\eqref{eq:CDP op TV}, we need to perform another conjugate transform. 
In particular, we need the value of $\lftc{\phi}$ at $\dynx(x)$ for $x \in \setd{X}$. 
Here, we use the approximation $\lftdd{\varphi}: \setg{Z} \ra \R$ in \eqref{eq:d-CDP op phi conj}, by applying LLT to the data points $\disc{\varphi}: \setg{Y} \ra \R$ for a properly chosen grid $\setg{Z} \subset \R^n$. Finally, we use the LERP extension $\llerp{\lftdd{\varphi}}$ of $\lftdd{\varphi}$ to approximate $\lftd{\varphi}$ at the required point $\dynx(x)$ for each $x \in \setd{X}$, and compute $\dcdpo \disc{J}$ in \eqref{eq:d-CDP op TV} as an approximation of $\cdpo J$ in \eqref{eq:CDP op TV}. 
\end{itemize}
With these approximations, we can introduce the \emph{discrete} CDP (d-CDP) operator as follows 
\begin{subequations} \label{eq:d-CDP op}
\begin{align}
& \disc{\varepsilon}(x) \Let \gamma \cdot \sum_{w \in \setd{W}} p(w) \cdot \lerp{\disc{J}} (x+w), & x \in \setd{X}, \label{eq:d-CDP op eps}  \\
&\lftdd{\varepsilon}(y) = \max_{x \in \setd{X}} \left\{ \inner{x}{y} - \disc{\varepsilon}(x) \right\},& y \in \setg{Y}, \label{eq:d-CDP op eps conj} \\
&\lftdd{\costu}(v) = \max_{u \in \setd{U}} \left\{ \inner{u}{v} - \disc{\costu}(u) \right\},& v \in \setg{V}, \label{eq:d-CDP op Ci conj} \\
& \disc{\varphi}(y) \Let \llerp{\lftdd{\costu}}(-B^{\top} y) + \lftdd{\varepsilon}(y), & y \in \setg{Y}, \label{eq:d-CDP op phi} \\
&\lftdd{\varphi}(z) = \max_{y \in \setg{Y}} \left\{ \inner{y}{z} - \disc{\varphi}(y) \right\},& z \in \setg{Z}, \label{eq:d-CDP op phi conj} \\
& \dcdpo \disc{J}(x) \Let \costx(x) + \llerp{\lftdd{\varphi}} \big( \dynx(x) \big), &  x \in \setd{X}. \label{eq:d-CDP op TV}
\end{align}
\end{subequations} 

The proper construction of the grids $\setg{Y}$, $\setg{V}$, and $\setg{Z}$ will be discussed in Section~\ref{subsec:grid construction}. 
We finish this subsection with the following remarks on the modification of the proposed algorithm for two special cases.

\begin{Rem}[Deterministic systems] For deterministic systems, i.e., $\dynw(x,u,w) = \dyn (x,u)$, we do not need to compute any expectation. 
Then, the operation in~\eqref{eq:d-CDP op eps} becomes the simple scaling $\disc{\varepsilon} = \gamma \cdot \disc{J}$.

\end{Rem}

\begin{Rem}[Analytically available $\lftc{\costu}$]
If the conjugate $\lftc{\costu}$ of the input cost is analytically available, we can use it directly in~\eqref{eq:d-CDP op phi} instead of $\llerp{\lftdd{\costu}}$ and avoid the corresponding approximation; i.e., there is no need for construction of $\setg{V}$ and the computation of $\lftdd{\costu}$ in~\eqref{eq:d-CDP op Ci conj}.
\end{Rem}

\subsection{Analysis of ConjVI algorithm} \label{subsec:analysis} 

We now  provide our main theoretical results concerning the convergence, complexity, and error of the proposed algorithm. 
Let us begin by presenting the assumptions to be called in this subsection.

\begin{As}[Grids]\label{As:grids} 

Consider the following properties for the grids in Algorithm~\ref{alg:d-CDP separ cost} (consult the Notations in Section~\ref{sec:intro}):

\begin{enumerate}[label=(\roman*)]

\item \label{As:grid V} The grid $\setg{V}$ is constructed such that $\co (\setsg{V}) \supseteq \setc{L} (\disc{\costu})$. 
\item \label{As:grid Z} The grid $\setg{Z}$ is constructed such that $\co (\setg{Z}) \supseteq \dynx\big(\setd{X}\big)$. 

\item \label{As:grid complexity} The construction of $\setg{Y}$, $\setg{V}$, and $\setg{Z}$ requires at most $\ord(X+U)$ operations. 
The cardinality of the grids $\setg{Y}$ and $\setg{Z}$ (resp. $\setg{V}$) in each dimension is the same as that of $\setd{X}$ (resp. $\setd{U}$) in that dimension so that $Y,Z = X$ and $V=U$.  

\end{enumerate}
\end{As}

\begin{As}[Extension operator]\label{As:extension op}
Consider the following properties for the operator~$\lerp{[\cdot]}$ in~\eqref{eq:d-CDP op eps}:

\begin{enumerate}[label=(\roman*)]

\item \label{As:ext nonexp} The extension operator is non-expansive w.r.t. the infinity norm;  
that is, for two discrete functions $\disc{J}_{i}:\setd{X} \ra \R \ (i=1,2)$ and their extensions $\lerp{\disc{J}_{i}}: \setc{X} \ra \R$, we have
$
\Vert\lerp{\disc{J}_{1}}-\lerp{\disc{J}_{2}}\Vert_{\infty} \leq \Vert\disc{J}_1-\disc{J}_2\Vert_{\infty}
$.

\item \label{As:ext error} Given a function $J:\setc{X} \ra \R$ and its discretization $\disc{J}:\setd{X} \ra \R$, the error of the extension operator is uniformly bounded, that is, $\Vert J - \lerp{\disc{J}}\Vert_{\infty} \leq e_\mathrm{e}$ for some constant $e_\mathrm{e} \geq 0$. 

\end{enumerate}
\end{As}


Our first result concerns the contractiveness of the d-CDP operator. 

\begin{Thm} [Convergence] \label{thm:convergence}
Let Assumptions~\ref{As:grids}-\ref{As:grid Z} and \ref{As:extension op}-\ref{As:ext nonexp} hold. 
Then, the d-CDP operator~\eqref{eq:d-CDP op} is $\gamma$-contractive w.r.t. the infinity-norm.
\end{Thm} 

The preceding theorem implies that the approximate ConjVI Algorithm~\ref{alg:d-CDP separ cost} is indeed convergent given that the required conditions are satisfied. 
In particular, for deterministic dynamics, $\co (\setg{Z}) \supseteq \dynx\big(\setd{X}\big)$ is sufficient for Algorithm~\ref{alg:d-CDP separ cost} to be convergent. 
We next consider the time complexity of our algorithm. 

\begin{Thm} [Complexity] \label{thm:complexity}
Let Assumption~\ref{As:grids}-\ref{As:grid complexity} hold. 
Also assume that each evaluation of the extension operator~$\lerp{[\cdot]}$ in~\eqref{eq:d-CDP op eps} requires $\ord (E)$ operations. 
Then, the time complexities of initialization and each iteration in Algorithm~\ref{alg:d-CDP separ cost} are of $\ord (X+U)$ and $ \wt{\ord}(XWE)$, respectively.
\end{Thm}

The requirements of Assumption~\ref{As:grids}-\ref{As:grid complexity} will be discussed in Section~\ref{subsec:grid construction}. 
Recall that each iteration of VI (in primal domain) has a complexity of $\ord (XUWE)$, 
where $E$ denotes the complexity of the extension operation used in~\eqref{eq:d-DP op}.  
This observation points to a basic characteristic of the proposed approach: 
ConjVI reduces the quadratic complexity of VI to a linear one by replacing the minimization operation in the primal domain with a simple addition in the conjugate domain. 
Hence, for the problem class of Assumption~\ref{As:prob data}, ConjVI is expected to lead to a reduction in the computational cost. 
We note that ConjVI, like VI and other approximation schemes that utilize discretization/abstraction of the continuous state and input spaces, still suffers from the so-called ``curse of dimensionality.'' 
This is because the sizes $X$ and $U$ of the discretizations increase exponentially with the dimensions $n$ and $m$ of the corresponding spaces. 
However, for ConjVI, this exponential increase is of rate $\max \{ m,n \}$, compared to the rate $m+n$ for VI.

Let us also note that the most crucial step that allows the speedup discussed above is the \emph{interpolative discrete conjugation} in \eqref{eq:d-CDP op TV} that approximates $\lftdd{\varphi}$ at the point $\dynx(x)$. 
In this regard, notice that we can alternatively compute $\lftdd{\varphi}\big( \dynx(x) \big) = \max_{y \in \setg{Y}} \left\{ \inner{y}{\dynx(x)} - \disc{\varphi}(y) \right\} $ exactly via enumeration over $y \in \setg{Y}$ for each $x \in \setd{X}$ (then, the computation of $\lftdd{\varphi}:\setg{Z} \ra \R$ in \eqref{eq:d-CDP op phi conj} is not needed anymore).  
However, this approach requires $\ord (XY) = \ord(X^2)$ operations in the last step, hence rendering the proposed approach computationally impractical. 
Of course, the application of interpolative discrete conjugation has its cost: 
The LERP extension in \eqref{eq:d-CDP op TV} can lead to non-convex outputs (even if Assumption~\ref{As:convex} holds true). 
This, in turn, can introduce a dualization error. 
We finish with the following result on the error of the proposed ConjVI algorithm. 

\begin{Thm}[Error] \label{thm:error}
Let Assumptions~\ref{As:convex}, \ref{As:grids}-\ref{As:grid V}\&\ref{As:grid Z}, and  \ref{As:extension op}-\ref{As:ext nonexp} hold. 
Consider the true optimal value function $J\opt = \dpo J\opt : \setc{X} \ra \R$ and its discretization $\disc{J\opt}: \setd{X} \ra \R$, and let Assumption~\ref{As:extension op}-\ref{As:ext error} hold for $J\opt$. 
Also, let $\disc{\wh{J}} : \setd{X} \ra \R$ be the output of Algorithm~\ref{alg:d-CDP separ cost}. 
Then, 
\begin{equation} \label{eq:error}
\Vert\disc{\wh{J}} - \disc{J\opt}\Vert_{\infty} \leq \frac{\gamma (e_\mathrm{e} + e_\mathrm{t})  + e_\mathrm{d}}{1-\gamma},
\end{equation}
where $e_\mathrm{d} = e_\mathrm{u} + e_\mathrm{v} + e_\mathrm{x} + e_\mathrm{y} + e_\mathrm{z}$, and
\begin{subequations} \label{eq:error terms}
\begin{align}
e_\mathrm{u} &= c_\mathrm{u} \cdot  \dish (\setc{U}, \setd{U}),  \\
e_\mathrm{v} &= c_\mathrm{v} \cdot \dish \big(\co( \setg{V}), \setg{V}\big), \label{eq:error term v} \\
e_\mathrm{x} &=   c_\mathrm{x} \cdot  \dish \big(\setc{X}, \setd{X}\big),\\ 
e_\mathrm{y} &= c_\mathrm{y} \cdot \max_{x \in \setd{X}} \dist\big(\partial  (J\opt -\costx) (x),\setg{Y}\big), \label{eq:error term y}\\ 
e_\mathrm{z} &= c_\mathrm{z} \cdot \dish \big(\dynx(\setd{X}), \setg{Z}\big), \label{eq:error term z}
\end{align}
\end{subequations}
with constants $c_\mathrm{u}, c_\mathrm{v}, c_\mathrm{x}, c_\mathrm{y}, c_\mathrm{z} > 0$ depending on the problem data.
\end{Thm}

Let us first note that Assumption~\ref{As:convex} implies that the DP and CDP operators preserve convexity, and they both have the true optimal value function $J\opt$ as their fixed point (i.e., the duality gap is zero). 
Otherwise, the proposed scheme can suffer from large errors due to dualization.
Moreover, Assumptions~\ref{As:grids}-\ref{As:grid V}\&\ref{As:grid Z} on the grids $\setg{V}$ and $\setg{Z}$ are required for bounding the error of approximate discrete conjugations using LERP in \eqref{eq:d-CDP op phi} and \eqref{eq:d-CDP op TV}; see the proof of Lemmas~\ref{lem:phi error} and \ref{lem:error d-CDP}. 
The remaining sources of error in the proposed approximate implementation of ConjVI are captured by the three error terms in \eqref{eq:error}: 
\begin{itemize}
\item[(i)] $e_\mathrm{e}$ is due to the approximation of the value function using the extension operator~$\lerp{[\cdot]}$; 
\item[(ii)] $e_\mathrm{t}$ corresponds to the termination of the algorithm after a finite number of iterations;
\item[(iii)] $e_\mathrm{d}$ captures the error due to the discretization of the primal and dual state and input domains.
\end{itemize}

We again finish with the following remarks on the modification of the proposed algorithm for deterministic systems and analytically available $\lftc{\costu}$.

\begin{Rem}[Deterministic systems] If the dynamics are deterministic, then the complexity of each iteration of Algorithm~\ref{alg:d-CDP separ cost} reduces to $ \wt{\ord}(X)$.
Moreover, in this case, the error term $e_\mathrm{e}$ disappears.
\end{Rem}

\begin{Rem}[Analytically available $\lftc{\costu}$]
If the conjugate $\lftc{\costu}$ of the input cost is analytically available and used in~\eqref{eq:d-CDP op phi} instead of the LERP extension $\llerp{\lftdd{\costu}}$, the error term due to discretization modifies to $e_\mathrm{d} = e_\mathrm{x} + e_\mathrm{y} + e_\mathrm{z}$. 
That is, the error terms $e_\mathrm{u}$ and $e_\mathrm{v}$ corresponding to the discretization of the primal and dual input spaces disappear.
\end{Rem}

\subsection{Construction of the grids}\label{subsec:grid construction} 

In this subsection, we provide specific guidelines for the construction of the grids $\setg{Y}$, $\setg{V}$ and $\setg{Z}$. 
We note that these discrete sets must be \emph{grid-like} since they form the dual grid for the three conjugate transforms that are handled using LLT. 
The presented guidelines aim to minimize the error terms in \eqref{eq:error terms} while taking into account the properties laid out in Assumption~\ref{As:grids}. 
In particular, the schemes described below satisfy the requirements of Assumption~\ref{As:grids}-\ref{As:grid complexity}.

\noindent\textbf{Construction of $\setg{V}$.} Assumption~\ref{As:grids}-\ref{As:grid V} and the error term $e_\mathrm{v}$ in \eqref{eq:error term v} 
suggest that we find the smallest input dual grid $\setg{V}$ such that $\co (\setsg{V}) \supseteq \setc{L} (\disc{\costu})$. 
This latter condition essentially means that $\setg{V}$ must ``more than cover the range of slope'' of the function $\disc{\costu}$; recall that $\setc{L} (\disc{\costu}) = \Pi_{j=1}^m \left[\lip_j^- (\disc{\costu}), \lip_j^- (\disc{\costu})\right]$, where $\lip_j^- (\disc{\costu})$ (resp. $\lip_j^+ (\disc{\costu})$) is the minimum (resp. maximum) slope of $\disc{\costu}$ along the $j$-th dimension. 
Hence, we need to compute/approximate $\lip_j^\pm (\disc{\costu})$ for $j=1,\ldots,m$. 
A conservative approximation is $\lip_j^- (\costu) = \min \nicefrac{\partial \costu}{\partial u_j}$ and $\lip_j^+ (\costu) = \max \nicefrac{\partial \costu}{\partial u_j}$, assuming $\costu$ is differentiable. 
Alternatively, we can directly use the discrete input cost $\disc{\costu}$ for computing $\lip_j^\pm (\disc{\costu})$. 
In particular, if the domain~$\setd{U} = \setg{U} = \Pi_{j=1}^{m} \setg{U}_{j}$ of $\disc{\costu}$ is grid-like and $\costu$ is convex,  
we can take $\lip_j^- (\disc{\costu})$ (resp. $\lip_j^+ (\disc{\costu})$) to be the minimum first forward difference (resp. maximum last backward difference) of $\disc{\costu}$ along the $j$-th dimension (this scheme requires $\ord (U)$ operations). 
Having $\lip_j^\pm (\disc{\costu})$ at our disposal, we can then construct $\setsg{V} = \Pi_{j=1}^{m} {\setsg{V}}_{j}$ such that, in each dimension $j$, ${\setsg{V}}_{j}$ is uniform and has the same cardinality as $\setg{U}_j$, and $\co ({\setsg{V}}_{j}) = \left[\lip_j^- (\disc{\costu}),  \lip_j^+ (\disc{\costu})\right]$. 
Finally, we construct $\setg{V}$ by extending $\setsg{V}$ uniformly in each dimension (by adding a smaller and a larger element to $\setsg{V}$ in each dimension while preserving the resolution in that dimension).  

\noindent\textbf{Construction of $\setg{Z}$.} According to Assumption~\ref{As:grids}-\ref{As:grid Z}, the grid $\setg{Z}$ must be constructed such that $\co (\setg{Z}) \supseteq \dynx\big(\setd{X}\big)$. 
This can be simply done by finding the vertices of the smallest box that contains the set $\dynx\big(\setd{X}\big)$. 
Those vertices give the diameter of $\setg{Z}$ in each dimension. 
We can then, for example, take $\setg{Z}$ to be the uniform grid with the same cardinality as $\setg{Y}$ in each dimension (so that $Z=Y$). 
This way, $$\dish \big(\dynx(\setd{X}), \setg{Z}\big) \leq \dish \big(\co(\setg{Z}), \setg{Z}\big),$$ 
and hence $e_\mathrm{z}$ in \eqref{eq:error term z} reduces by using finer grids $\setg{Z}$. 
This construction has a time complexity of $\ord (X)$. 

\noindent\textbf{Construction of $\setg{Y}$.} Construction of the state dual grid $\setg{Y}$ is more involved. 
According to Theorem~\ref{thm:error}, we need to choose a grid that minimizes $e_\mathrm{y}$ in \eqref{eq:error term y}. 
This can be done by choosing $\setg{Y}$ such that $\setg{Y} \cap \partial (J\opt - \costx) \neq \emptyset$ for all $x \in \setd{X}$ so that $e_\mathrm{y} = 0$. 
Even if we had access to the optimal value function $J\opt$, satisfying such a condition could lead to a dual grid $\setg{Y} \subset \R^n$ of size $\ord (X^n)$. 
Such a large size violates Assumption~\ref{As:grids}-\ref{As:grid complexity} on the size of $\setg{Y}$, 
and essentially renders the proposed algorithm impractical for dimensions $n \geq 2$. 
A more practical condition is $\co (\setg{Y}) \cap \partial (J\opt - \costx) \neq \emptyset$ for all $x \in \setd{X}$ so that 
$$\max_{x \in \setd{X}} \dist\big(\partial  (J\opt -\costx) (x),\setg{Y}\big) \leq \dish\big(\co(\setg{Y}),\setg{Y}\big),$$ 
and hence $e_\mathrm{y}$ reduces by using a finer grid $\setg{Y}$. 
The latter condition is satisfied if $\co (\setg{Y}) \supseteq \setc{L} (J\opt - \costx)$, i.e., if $\co (\setg{Y})$ ``covers the range of slope'' of $(J\opt - \costx)$. 
Hence, we need to approximate the range of slope of $(J\opt - \costx)$. 
To this end, we first use the fact that $J\opt$ is the fixed point of DP operator~\eqref{eq:DP op} to approximate $\rng(J\opt - \costx)$ by 
$$
R = \frac{\rng (\disc{\costu})+\gamma \cdot \rng (\disc{\costx})}{1-\gamma}
.$$ 
We then construct the gird $\setg{Y} = \Pi_{i=1}^{n} \setg{Y}_{i}$ such that, for each dimension $i$, we have 
\begin{equation}\label{eq:condition}
\pm  \frac{\alpha R}{\diam{\setd{X}}^{i}} \in \co(\setg{Y}_i)
\end{equation}
where $\diam{\setd{X}}^{i}$ denotes the diameter of the projection of $\setd{X}$ on the $i$-th dimension. 
Here, the coefficient $\alpha > 0$ is a scaling factor mainly depending on the dimension of the state space. In particular, by setting $\alpha = 1$, the value $\nicefrac{R}{\diam{\setd{X}}^{i}}$ is the slope of a linear function with range $R$ over the domain $\diam{\setd{X}}^{i}$.     
This construction has a one-time computational cost of $\ord(X+U)$ for computing $\rng (\disc{\costu})$ and $\rng (\disc{\costx})$. 

\noindent\textbf{Dynamic construction of $\setg{Y}$.}
Alternatively, we can construct $\setg{Y}$ \emph{dynamically} at each iteration
to minimize the corresponding error in each application of the d-CDP operator given by  (see Lemma~\ref{lem:phi error 2} and Proposition~\ref{prop:error d-CDP})
$$
e_\mathrm{y}  = c_\mathrm{y} \cdot \max_{x \in \setd{X}} \dist\big(\partial  (\dpo J-\costx) (x),\setg{Y}\big).$$ 
This means that line~\ref{line_alg2:const Y} in Algorithm~\ref{alg:d-CDP separ cost} is moved inside the iterations, after line~\ref{line_alg2:iter 1}. 
Similar to the static scheme described above, the aim here is to construct $\setg{Y}$ such that $\co (\setg{Y}) \supseteq \setc{L} (\dpo J - \costx)$. 
Since we do not have access to $\dpo J$ (it is the output of the current iteration), 
we can again use the definition of the DP operator~\eqref{eq:DP op} to approximate $\rng(\dpo J - \costx)$ by 
$$R = \rng (\disc{\costu})+ \gamma \cdot \rng (\disc{J}),$$ 
where $\disc{J}$ is the output of the previous iteration. 
We then construct the gird $\setg{Y} = \Pi_{i=1}^{n} \setg{Y}_{i}$ such that, for each dimension $i$, the condition~\eqref{eq:condition} holds. 
This construction has a one-time computational cost of $\ord(U)$ for computing $\rng (\disc{\costu})$ and a per iteration computational cost of $\ord(X)$ for computing $\rng (\disc{J})$. 
Notice, however, that under this dynamic construction, the error bound of Theorem~\ref{thm:error} does not hold true. 
More importantly, with a dynamic grid $\setg{Y}$ that varies in each iteration, there is no guarantee for ConjVI to converge.

\section{Numerical simulations}
\label{sec:numerical ex}
In this section, we compare the performance of the proposed ConjVI algorithm with the benchmark VI algorithm (in primal domain) through three numerical examples. 
For the first example, we focus on a synthetic system satisfying the conditions of assumptions considered in this study to examine our theoretical results. 
We then showcase the application of ConjVI in solving the optimal control problem of an inverted pendulum and a batch reactor. 
The simulations were implemented via MATLAB version R2017b, on a PC with an Intel Xeon 3.60~GHz processor and 16~GB RAM.  
We also provide the ConjVI MATLAB package~\cite{Kolari21ConjVIMAT} for the implementation of the proposed algorithm. 
The package also includes the numerical simulations of this section. 
We note that multiple routines in the developed package are borrowed from the d-CDP MATLAB package~\cite{Kolari21dCDPMAT}. 
Also, for the discrete conjugation (LLT), we used the MATLAB package (in particular, the \texttt{LLTd} routine) provided in \cite{Lucet97}. 

\subsection{Example 1 -- Synthetic}
\label{subsec:numerical ex 1}

We consider the linear system $x^+ = Ax+Bu+w$ with $A = [2 \ \ 1; \ 1 \ \ 3]$, $B = [1 \ \ 1; \ 1 \ \ 2]$. 
The problem of interest is the infinite-horizon, optimal control of this system with cost functions $\costx(x) = 10\norm{x}_2^2$ and $\costu(u) = e^{|u_1|} + e^{|u_2|} - 2$, and discount factor $\gamma = 0.95$. 
We consider state and input constraint sets~$\setc{X} = [-1,1]^2$ and $\setc{U} = [-2,2]^2$, respectively. 
The disturbance is assumed to have a uniform distribution over the finite support $\setd{W} = \{0,\pm 0.05 \}\times \{0 \}$ of size $W=3$. 
Notice how the stage cost is a combination of a quadratic term (in state) and an exponential term (in input). 
Particularly, the control problem at hand does not have a closed-form solution. 
We use uniform, grid-like discretizations $\setg{X}$ and $\setg{U}$ for the state and input spaces such that $\co (\setg{X})= \setc{X}$ and $\co (\setg{U})= \setc{U}$. 
This choice allows us to deploy \emph{multilinear interpolation}, which is non-expansive, as the extension operator $\lerp{[\cdot]}$ in the d-DP operation~\eqref{eq:d-DP op} in VI, and in the d-CDP operation~\eqref{eq:d-CDP op eps} in ConjVI. 
The grids $\setg{V}, \setg{Z} \subset \R^2$ are also constructed uniformly, following the guidelines provided in Section~\ref{subsec:alg}. 
For the construction of $\setg{Y} \subset \R^2$, we also follow the guidelines of Section~\ref{subsec:alg} with $\alpha = 1$. 
In particular, we also consider the \emph{dynamic} scheme for the construction of $\setg{Y}$ in ConjVI (hereafter, referred to as ConjVI-d). 
Moreover, in each implementation of VI and ConjVI(-d), all of the involved grids ($\setg{X}, \setg{U}, \setg{Y}, \setg{V}, \setg{Z}$) are chosen to be of the same size~$N^2$ (with $N$ points in each dimension). 
We are particularly interested in the performance of these algorithms, as $N$ increases. 
We note that the described setup satisfies all of the assumptions in this study. 

The results of our numerical simulations are shown in Figure~\ref{fig:ex 1s}. 
As shown in Figures~\ref{fig:ex 1s conv}, both VI and ConjVI are indeed convergent with a rate less than or equal to the discount factor $\gamma = 0.95$; see Theorem~\ref{thm:convergence}. 
In particular, ConjVI terminates in $k_\mathrm{t} = 55$ iterations, 
compared to $k_\mathrm{t} =102$ iterations required for VI to reach the termination bound $e_\mathrm{t} = 0.001$. 
Not surprisingly, this faster convergence, combined with the lower time complexity of ConjVI in each iteration, leads to a significant reduction in the running time of this algorithm compared to VI. 
This effect can be seen in Figure~\ref{fig:ex 1s runtime}, where the run-time of ConjVI for $N=41$ is an order of magnitude less than that of VI for $N=11$. 
In this regard, we note that the setting of this numerical example leads to $\ord(k_\mathrm{t}N^4W)$ and $\ord(k_\mathrm{t}N^2W)$ time complexities for VI and ConjVI, respectively; see Theorem~\ref{thm:complexity} and the discussion after that. 
Indeed, the running times in Figure~\ref{fig:ex 1s runtime} match these complexities. 

Since we do not have access to the true optimal value function, in order to evaluate the performance of the outputs of the VI and ConjVI, we consider the performance of the greedy policy 
$$\mu(x) \in \argmin_{u \in \setc{U}(x) \cap \setg{U}} \big\{  \cost(x,u) + \gamma \cdot \EE_w \llerp{\disc{J}}\big(g(x,u,w)\big) \big\}, 
$$ 
w.r.t. the discrete value function $\disc{J}$ computed using these algorithms 
(we note that, for finding the greedy action, we used the same discretization $\setg{U}$ of the input space and the same extension $\llerp{\disc{J}}$ of the value function as the one used in VI and ConjVI, however, this need not be the case in general). 
Figure~\ref{fig:ex 1s cost} reports the average cost of one hundred instances of the optimal control problem with greedy control actions. 
As shown, the reduction in the run-time in ConjVI comes with an increase in the cost of the controlled trajectories. 

Let us now consider the effect of \emph{dynamic} construction of the state dual grid $\setg{Y}$. 
As can be seen in Figure~\ref{fig:ex 1s conv}, using a dynamic $\setg{Y}$ leads to a slower convergence (ConjVI-d terminates in $k_\mathrm{t}=100$ iterations). 
We note that the relative behavior of the convergence rates in Figures~\ref{fig:ex 1s conv} was also seen for other grid sizes in the discretization scheme.
However, we see a small increase in the running time of ConjVI-d compared to ConjVI since the per iteration complexity for ConjVI-d is again of $\ord(k_\mathrm{t}N^2W)$; see Figure~\ref{fig:ex 1s runtime}. 
More importantly, as depicted in Figure~\ref{fig:ex 1s cost}, ConjVI-d shows almost the same performance as VI when it comes to the quality of the greedy actions. 
This is because the dynamic construction of $\setg{Y}$ in ConjVI-d uses the available computational power (related to the size of the discretization) smartly by finding the smallest grid $\setg{Y}$ in each iteration, to minimize the error of that same iteration. 

We note that our simulations show that for the \emph{deterministic} system, ConjVI-d has a similar converge rate as ConjVI. 
This effect can be seen in Figure~\ref{fig:ex 1d conv}, where ConjVI-d terminates in 10 iterations. 
Interestingly, in this particular example, ConjVI converges to the fixed point after 7 iterations ($\disc{J}_{8} = \dcdpo \disc{J}_{7}$) for the deterministic system. 
Let us finally note that the conjugate $\lftc{\costu}$ of the input cost in the provided example is indeed analytically available. 
One can use this analytic representation to exactly compute $\lftc{\costu}$ in \eqref{eq:d-CDP op TV} and avoid the corresponding numerical approximation. 
With such a modification, the computational cost reduces, however, our numerical experiments show that for the provided example, the ConjVI outputs effectively the same value function within the same number of iterations (results are not shown here).

\begin{figure}[t]
\begin{subfigure}{.33\textwidth}
  \centering
  \includegraphics[width=1\linewidth]{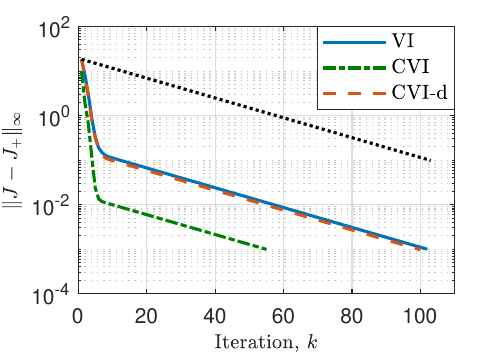}
  \caption{}
  \label{fig:ex 1s conv}
\end{subfigure}%
\begin{subfigure}{.33\textwidth}
  \centering
  \includegraphics[width=1\linewidth]{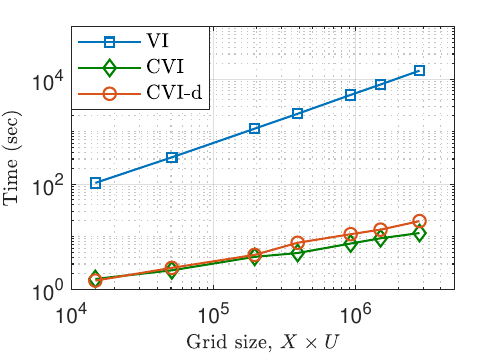}
  \caption{}
  \label{fig:ex 1s runtime}
\end{subfigure}%
\begin{subfigure}{.33\textwidth}
  \centering
  \includegraphics[width=1\linewidth]{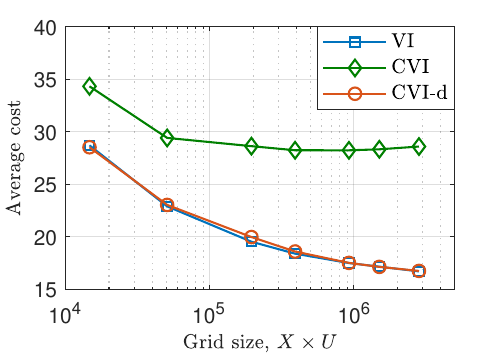}
  \caption{}
  \label{fig:ex 1s cost}
\end{subfigure}
\caption{VI vs. ConjVI (CVI) -- synthetic example with \emph{stochastic} dynamics $x^+ = Ax+Bu+w$: 
(a)~Convergence rate for $N=41$; (b)~Running time; (c)~Average cost of one hundred instances of the control problem with random initial conditions over $T=100$ time steps. 
The black dashed-dotted line in (a) corresponds to exponential convergence with coefficient $\gamma = 0.95$. 
CVI-d corresponds to \emph{dynamic} construction of the dual grid $\setg{Y}$ in the ConjVI algorithm.}
\label{fig:ex 1s}
\end{figure} 

\begin{figure}[t]
\includegraphics[width=.4\linewidth]{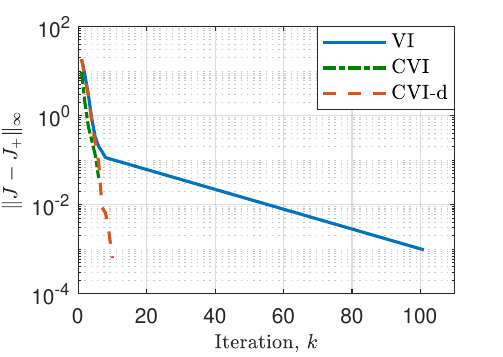}
\caption{Convergence of VI and ConjVI with \emph{deterministic} dynamics $x^+=Ax+Bu$; cf. Figure~\ref{fig:ex 1s conv}.}
\label{fig:ex 1d conv}
\end{figure}

\subsection{Example 2 -- Inverted pendulum}
\label{subsec:numerical ex 2}

We use the setup (model and stage cost) of~\cite[App.~C.2.2]{Kolari21dCDParxiv} with discount factor $\gamma = 0.95$. 
In particular, the state and input costs are both quadratic ($\norm{\cdot}_2^2$), and the discrete-time, nonlinear dynamics is of the form $x^+ = \dynx(x) + B u + w$, where
$$\dynx(x_1,x_2) = \left[ \begin{array}{c} x_1+\alpha_{12} x_2 \\ \alpha_{21} \sin x_1 + \alpha_{22} x_2 \end{array} \right],\ B = \left[ \begin{array}{c} 0 \\ \beta \end{array} \right], \quad (\alpha_{12}, \alpha_{21}, \alpha_{22}, \beta \in \R).$$
State and input constraints are described by $\setc{X} = [-\frac{\pi}{3},\frac{\pi}{3}]\times[-\pi,\pi]\subset \R^2$ and $\setc{U} = [-3,3]\subset \R$. 
The disturbance has a uniform distribution over the finite support $\setg{W} = \{0, \pm0.025\frac{\pi}{3}, \pm0.05\frac{\pi}{3} \} \times \{0, \pm0.025\pi, \pm0.05\pi \} \subset \R^2$ of size $W=5^2$. 
We use uniform, grid-like discretizations $\setg{X}$ and $\setg{U}$ for the state and input spaces such that $\co (\setg{X})= [-\frac{\pi}{4},\frac{\pi}{4}]\times[-\pi,\pi]\subset \setc{X}$ and $\co (\setg{U})= \setc{U}$. 
This choice of discrete state space $\setg{X}$ particularly satisfies the feasibility condition of Assumption~\ref{As:feasible discrete}. 
(Note however that the set $\setc{X}$ does not satisfy the feasibility condition of Assumption~\ref{As:prob data}-\ref{As:constr}). 
Also, we use \emph{nearest neighbor} extension (which is non-expansive) for the extension operators in~\eqref{eq:d-DP op} for VI and in~\eqref{eq:d-CDP op eps} for ConjVI. 
The grids $\setg{V} \subset \R$ and  $\setg{Z}, \setg{Y} \subset \R^2$ are also constructed uniformly, following the guidelines of Section~\ref{subsec:grid construction} (with $\alpha = 1$). 
We again also consider the \emph{dynamic} scheme for the construction of $\setg{Y}$.
Moreover, in each implementation of VI and ConjVI(-d) the termination bound is $e_\mathrm{t} = 0.001$, and all of the involved grids are chosen to be of the same size $N$ in each dimension, i.e., $X = Y = Z = N^2$ and $U = V = N$. 

The results of simulations are shown in Figures~\ref{fig:ex 2s} and \ref{fig:ex 2d conv}. 
As reported, we essentially observe the same behaviors as before. 
In particular, the application of ConjVI(-d), especially for deterministic dynamics, leads to faster convergence and a significant reduction in the running time; see Figures~\ref{fig:ex 2s conv}, \ref{fig:ex 2s runtime} and \ref{fig:ex 2d conv}. 
Note that Figure~\ref{fig:ex 2d conv} also shows the non-monotone behavior of ConjVI-d for scaling factor $\alpha = 3$. 
In this regard, recall that when the grid $\setg{Y}$ is constructed dynamically and varies at each iteration, the d-CDP operator is not necessarily contractive. 
Moreover, as shown in Figures~\ref{fig:ex 2s runtime} and \ref{fig:ex 2s cost}, this dynamic scheme leads to a huge improvement in the performance of the corresponding greedy policy at the expense of a small increase in the computational cost.

\begin{figure}[t]
\begin{subfigure}{.33\textwidth}
  \centering
  \includegraphics[width=1\linewidth]{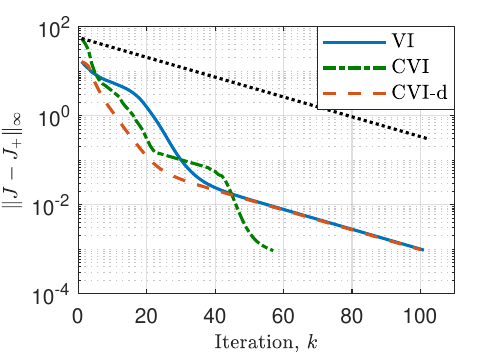}
  \caption{}
  \label{fig:ex 2s conv}
\end{subfigure}%
\begin{subfigure}{.33\textwidth}
  \centering
  \includegraphics[width=1\linewidth]{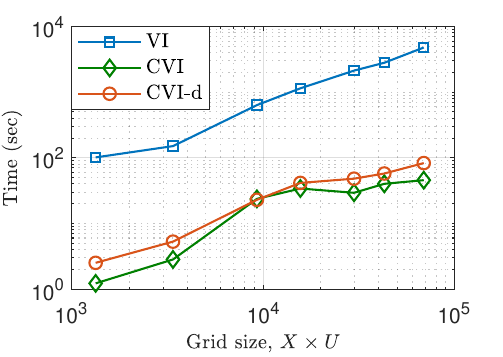}
  \caption{}
  \label{fig:ex 2s runtime}
\end{subfigure}%
\begin{subfigure}{.33\textwidth}
  \centering
  \includegraphics[width=1\linewidth]{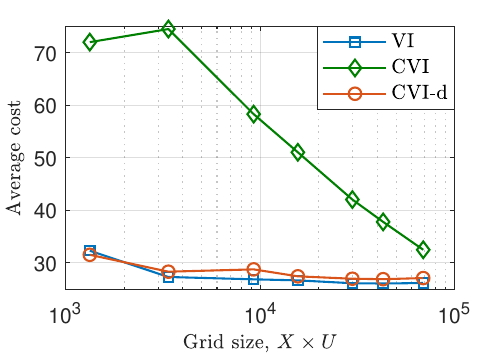}
  \caption{}
  \label{fig:ex 2s cost}
\end{subfigure}%
\caption{VI vs. ConjVI (CVI) -- optimal control of noisy inverted pendulum: 
(a)~Convergence rate for $N=41$; (b)~Running time; (c)~Average cost of one hundred instances of the control problem with random initial conditions over $T=100$ time steps. 
The black dashed-dotted line in (a) corresponds to exponential convergence with coefficient $\gamma = 0.95$. 
CVI-d corresponds to \emph{dynamic} construction of the dual grid $\setg{Y}$ in the ConjVI algorithm.}
\label{fig:ex 2s}
\end{figure} 

\begin{figure}[t]
\includegraphics[width=.4\linewidth]{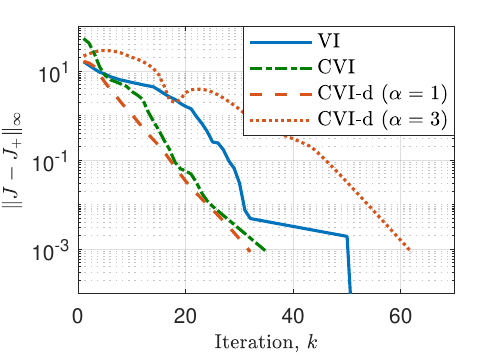}
\caption{Convergence of VI and ConjVI with \emph{deterministic} dynamics $x^+=\dynx (x) + Bu$; cf. Figure~\ref{fig:ex 2s conv}.}
\label{fig:ex 2d conv}
\end{figure} 

\subsection{Example 3 -- Batch Reactor}
\label{subsec:numerical ex 3}

Our last numerical example concerns the optimal control of a system with four states and two input channels, namely, an unstable batch reactor. 
The setup (dynamics, cost, and constraints) are borrowed from~\cite[Sec.~6]{Arman20}. 
In particular, we consider a \emph{deterministic} linear dynamics $x^+ = Ax+Bu$, with costs $\costx (x) = 2 \norm{x}_2^2$ and $ \costu (u) = \norm{u}_2^2$, 
discount factor $\gamma = 0.95$, 
and constraints $x \in \setc{X} = [-2,2]^4 \subset \R^4$ and $ u \in \setc{U} = [-2,2]^2 \subset \R^2$. 
Once again, we use uniform, grid-like discretizations $\setg{X}$ and $\setg{U}$ for the state and input spaces such that $\co (\setg{X})= [-1,1]^4 \subset \setc{X}$ and $\co (\setg{U})= \setc{U}$. 
The grids $\setg{V} \subset \R^2$ and  $\setg{Z}, \setg{Y} \subset \R^4$ are also constructed uniformly, following the guidelines of Section~\ref{subsec:grid construction} (with $\alpha = 1$). 
Moreover, in each implementation of VI and ConjVI, the termination bound is $e_\mathrm{t} = 0.001$ and all of the involved grids are chosen to be of the same size $N$ in each dimension, i.e., $X = Y = Z = N^4$ and $U = V = N^2$.  
Finally, we note that we use \emph{multi-linear interpolation and extrapolation} for the extension operator in~\eqref{eq:d-DP op} for VI. 
Due to the extrapolation, the extension operator is no longer non-expansive and hence the convergence of VI is not guaranteed. 
On the other hand, since the dynamics is deterministic, there is no need for extension in ConjVI (recall that the scaled expectation in~\eqref{eq:d-CDP op eps} in ConjVI reduces to the simple scaling $\disc{\varepsilon} = \gamma \cdot \disc{J}$ for deterministic dynamics),
and hence the convergence of ConjVI only requires $\co (\setg{Z}) \supseteq \dynx\big(\setg{X}\big)$ and is guaranteed.  

The results of our numerical simulations are shown in Figure~\ref{fig:ex 3}. 
Once again, we see the trade-off between the time complexity and the greedy control performance in VI and ConjVI. 
On the other hand, ConjVI-d has the same control performance as VI with an insignificant increase in running time compared to ConjVI. 
In Figure~\ref{fig:ex 3 conv}, we again observe the non-monotone behavior of ConjVI-d 
(the d-CDP operator is expansive in the first six iterations). 
The VI algorithm is also showing a non-monotone behavior, where for the first nine iterations the d-DP operation is actually expansive. 
As we noted earlier, this is because the multi-linear extrapolation operation used in extension is expansive. 

\begin{figure}[t]
\begin{subfigure}{.33\textwidth}
  \centering
  \includegraphics[width=1\linewidth]{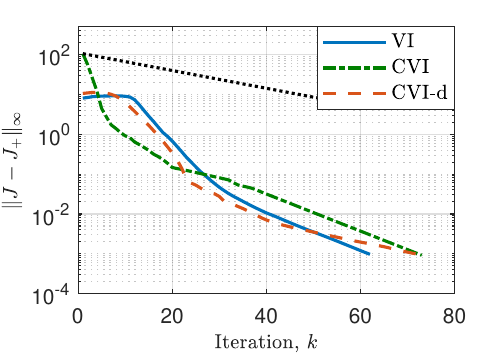}
  \caption{}
  \label{fig:ex 3 conv}
\end{subfigure}%
\begin{subfigure}{.33\textwidth}
  \centering
  \includegraphics[width=1\linewidth]{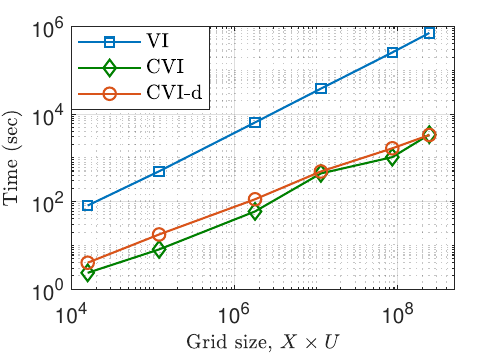}
  \caption{}
  \label{fig:ex 3 runtime}
\end{subfigure}%
\begin{subfigure}{.33\textwidth}
  \centering
  \includegraphics[width=1\linewidth]{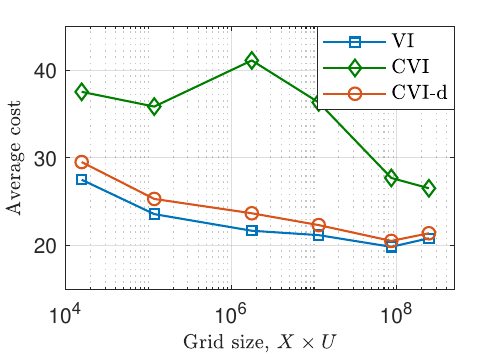}
  \caption{}
  \label{fig:ex 3 cost}
\end{subfigure}%
\caption{VI vs. ConjVI (CVI) -- optimal control of batch reactor: 
(a)~Convergence rate for $N=25$; (b)~Running time; (c)~Average cost of one hundred instances of the control problem with random initial conditions over $T=100$ time steps. 
The black dashed-dotted line in (a) corresponds to exponential convergence with coefficient $\gamma = 0.95$. 
CVI-d corresponds to \emph{dynamic} construction of the dual grid $\setg{Y}$ in the ConjVI algorithm.}
\label{fig:ex 3}
\end{figure}

\section{Final remarks}
\label{sec:conclusion}

In this paper, we proposed the ConjVI algorithm which reduces the time complexity of the VI algorithm from $\mathcal{O}(XU)$ to $\mathcal{O}(X+U)$. 
This better time complexity however comes at the expense of restricting the class of problem. 
In particular, there are two main conditions that must be satisfied in order to be able to apply the ConjVI algorithm: 
\begin{itemize}
\item[(i)] the dynamics must be of the form $x^+ = f_{\mathrm{s}}(x)+Bu+w$; and, 
\item[(ii)] the stage cost $C(x,u) = C_{\mathrm{s}}(x)+C_{\mathrm{i}}(u)$ must be separable.
\end{itemize}
Moreover, since ConjVI essentially solves the dual problem, for non-convex problems, it suffers from a non-zero duality gap. 
Based on our simulation results, we also notice a trade-off between computational complexity and control action quality: 
While ConjVI has a lower computational cost, VI generates better control actions. 
However, the dynamic scheme for the construction of state dual grid $\setg{Y}$ allows us to achieve almost the same performance as VI when it comes to the quality of control actions, with a small extra computational burden. 
In what follows, we provide our final remarks on the limitations of the proposed ConjVI algorithm and its relation to existing approximate VI algorithms.

\noindent\textbf{Relation to existing approximate VI algorithms.} 
The basic idea for complexity reduction introduced in this
study can be potentially combined with and further improve the existing sample-based VI algorithms. 
These sample-based algorithms solely focus on transforming the infinite-dimensional optimization in DP problems into computationally tractable ones, and in general, they have a time complexity of $\mathcal{O}(XU)$, depending on the product of the cardinalities of the discrete state and action spaces.
The proposed ConjVI algorithm, on the other hand, focuses on reducing this time complexity to $\mathcal{O}(X+U)$, by avoiding the minimization over input in each iteration. 
Take, for example, the aggregation technique in~\cite[Sec.~8.1]{Pow11} that leads to a piece-wise constant approximation of the value function. 
It is straightforward to combine ConjVI with this type of state space aggregation. 
Indeed, the numerical example of Section~\ref{subsec:numerical ex 2} essentially uses such aggregation by approximating the value function via nearest neighbor extension.

\noindent\textbf{Cost functions with a large Lipschitz constant.}
Recall that for the proposed ConjVI algorithm to be computationally efficient, the size $Y$ of the state dual grid $\mathbb{Y}^{\mathrm{g}}$ must be controlled by the size $X$ of the discrete state space $\setd{X}$ (Assumption~\ref{As:grids}-(iii)). 
Then, as the range of slope of the value function $J_\star$ increases, the corresponding error $e_\mathrm{y}$ in \eqref{eq:error term y} due to the discretization of the dual state space increases. 
The proposed dynamic approach for the construction of $\mathbb{Y}^{\mathrm{g}}$ partially addresses this issue by focusing on the range of slope of $J^{\mathrm{d}}_k$ in each iteration to minimize the discretization error of the same iteration $k$. 
However, when the cost function has a large Lipschitz constant, even this latter approach can fail to provide a good approximation of the value function. 
Table~\ref{tab:large L constant} reports the result of the numerical simulation of the unstable batch reactor with the stage cost 
\begin{equation}\label{eq:cost 2 BR}
C(x,u) = -\frac{4}{1+\eta} + \sum_{i=1}^4 \frac{1}{1+\eta-|x_i|} -\frac{2}{2+\eta} + \sum_{j=1}^2 \frac{1}{2+\eta-|u_j|},\quad \norm{x}_{\infty} \leq 1,\ \norm{u}_{\infty} \leq 2.
\end{equation}
Clearly, as $\eta \rightarrow 0$, we increase the range of slope of the cost function. 
As can be seen, the quality of the greedy action generated by ConjVI-d also deteriorates in this case.

\begin{table}
  \caption{VI vs. ConjVI - optimal control the batch reactor with stage cost~\eqref{eq:cost 2 BR} and $\eta = 0.01$. }
  \label{tab:large L constant}
  \centering
  \begin{tabular}{lcc}
    \toprule
    Algrithm     & Run-time (sec)    & Average cost (100 runs) \\
    \midrule
    VI & 7669 & 33.9     \\
    ConjVI & 55 & 73.5 \\
 	ConjVI-d & 90 & 74.0   \\
    \bottomrule
  \end{tabular}
\end{table}

\noindent\textbf{Gradient-based algorithms for solving the minimization over input.}
Let us first note that the minimization over $u$ in sample-based VI algorithms usually involves solving a difficult non-convex problem. 
This is particularly because the extension operation employed in these algorithms for approximating the value function using the sample points does not lead to a convex function in $u$ (e.g., take kernel-based approximations or neural networks). 
This is why in MDP and RL literature, it is quite common to consider a finite action space in the first place \cite{Busoniu17, Pow11}. 
Moreover, the minimization over $u$ again must be solved for each sample point in each iteration, while the application of ConjVI avoids solving this minimization in each iteration. 
In this regard, let us note that ConjVI uses a convex approximation of the value function, which allows for the application of a gradient-based algorithm for minimization over $u$ within the ConjVI algorithm. 
Indeed, in each iteration $k= 0,1,\ldots$, ConjVI solves (for deterministic dynamics)
$$J_{k+1}^{\mathrm{d}}(x) =  C_{\mathrm{s}}(x) +  \min_u \left\{ C_{\mathrm{i}}(u)+\gamma \cdot \max_{y \in \mathbb{Y}^{\mathrm{g}}} \left[ \left\langle y, f_{\mathrm{s}}(x) + Bu \right\rangle -  J^{\mathrm{d}*\mathrm{d}}_k(y) \right] \right\}, \quad x \in \mathbb{X}^{\mathrm{d}},$$
where 

$$J^{\mathrm{d}*\mathrm{d}}_k(y) = \max_{x \in \mathbb{X}^{\mathrm{d}}} \left\{ \left\langle x, y \right\rangle -  J^{\mathrm{d}}_k(x) \right\}, \quad y \in \mathbb{Y}^{\mathrm{g}},$$ 
is the discrete conjugate of the output of the previous iteration (computed using the LLT algorithm). 
Then, it is not hard to see that a subgradient of the objective of the minimization can be computed using $\mathcal{O} (Y)$ operations: 
for a given $u$, assuming we have access to the subdifferential $\partial C_{\mathrm{i}} (u)$, the subdifferential of the objective function is $\partial C_{\mathrm{i}} (u) + \gamma \cdot B^{\top} y_u$, where 
$$y_u \in \argmax_{y \in \mathbb{Y}^{\mathrm{g}}} \left\{ \left\langle y, f_{\mathrm{s}}(x) + Bu \right\rangle -  J^{\mathrm{d}*\mathrm{d}}_k(y) \right\}.$$ 
This leads to a per iteration time complexity of $\mathcal{O} (XY) = \mathcal{O} (X^2)$, which is again practically inefficient.

\appendix 

\section{Technical proofs}
\label{sec:proofs}

\subsection{Proof of Proposition~\ref{prop:CDP op conj}}
This result is an extension of \cite[Lem.~4.2]{Kolari21dCDParxiv} that accounts for the separable cost, the discount factor, and additive disturbance. 
Inserting the dynamics of Assumption~\ref{As:prob data}-\ref{As:dyn} into \eqref{eq:CDP op}, we can use the definition of conjugate transform to obtain 
(all the functions are extended to infinity outside their effective domains)
\begin{align*}
\cdpo J(x) - \costx(x) &= \max_{y} \ \min_{u, z} \left\{ \costu (u) + \gamma \cdot \EE_w J (z+w) + \inner{y}{\dynx (x) + Bu-z}  \right\} \\
&= \max_{y} \left\{ \inner{y}{\dynx(x)} - \max_{u} \left[ \inner{-B\tr y}{u} - \costu (u) \right] - \max_{ z} \left[ \inner{y}{z} -\gamma \cdot\EE_w J (z+w) \right]  \right\} \\
&= \max_{y } \left\{ \inner{y}{\dynx(x)} - \lftc{\costu}(-B\tr y) -\lftc{[\gamma \cdot\EE_w J (\cdot+w)]} (y)  \right\} \\
&= \max_{y } \left\{ \inner{y}{\dynx(x)} - \lftc{\costu}(-B\tr y) -\lftc{\epsilon} (y)  \right\} \\
&= \max_{y } \left\{ \inner{y}{\dynx(x)} - \phi (y)  \right\} \\
&= \lftc{\phi} \big( \dynx(x) \big),
\end{align*}
where we used the definition of $epsilon$ and $\phi$ in \eqref{eq:CDP op eps} and \eqref{eq:CDP op phi}, respectively. 

\subsection{Proof of Proposition~\ref{prop:CDP reform}}
We can use the representation~\eqref{eq:CDP op conj} and the definition of conjugate operation to obtain
\begin{align*}
\cdpo J(x) - \costx(x) &= \max_{y} \{ \inner{\dynx(x)}{y} - \phi(y) \} \\
& = \max_{y} \left\{ \inner{\dynx(x)}{y} - \lftc{\costu}(-B\tr y) - \lftc{\epsilon}(y) \right\} \\
& = \max_{y} \left\{ \inner{\dynx(x)}{y} - \bcc{[\lftc{\costu}]}(-B\tr y) - \lftc{\epsilon}(y) \right\} \\
& = \max_{y} \left\{ \inner{\dynx(x)}{y} - \max_{u \in \co(\setc{U})} \left[ \inner{-B\tr y}{u} - \bcc{\costu}(u) \right] - \lftc{\epsilon}(y) \right\} \\ 
& = \max_{y} \min_{u \in \co(\setc{U})} \left\{ \bcc{\costu}(u) +  \inner{y}{\dynx(x)+Bu} - \lftc{\epsilon}(y)  \right\},
\end{align*} 
where we used the fact that $\lftc{\costu}: \R^m \ra \R$ is proper, closed, and convex, and hence $\bcc{[\lftc{\costu}]} = \lftc{\costu}$. 
This follows from the fact that $\dom(\costu) = \setc{U}$ is assumed to be compact (Assumption~\ref{As:prob data}-\ref{As:constr}). 
Hence, the objective function of this maximin problem is convex in $u$, with $\co(\setc{U})$ being compact, which follows from convexity of $\bcc{\costu}: \co(\setc{U}) \ra \R$. 
Also, the objective function is concave in $y$, which follows from the convexity of $\lftc{\epsilon}$. 
Then, by Sion's Minimax Theorem (see, e.g., \cite[Thm.~3]{Simons95}), we have minimax-maximin equality, i.e.,  
\begin{align*}
\cdpo J(x) - \costx(x) & =  \min_{u} \max_{y} \left\{ \bcc{\costu}(u) +  \inner{y}{\dyn(x,u)} - \lftc{\epsilon}(y)  \right\} \nonumber \\
& = \min_{u } \left\{ \bcc{\costu}(u) + \max_{y} \big[ \inner{y}{\dyn(x,u)} - \lftc{\epsilon}(y) \big] \right\} \\ 
& = \min_{u} \left\{ \bcc{\costu}(u) + \bcc{\epsilon}\big( \dyn(x,u) \big) \right\} \\
& = \min_{u} \left\{ \bcc{\costu}(u) + \gamma \cdot \bcc{[\EE_w J (\cdot+w)]}\big( \dyn(x,u) \big) \right\},
\end{align*}
where the last equality, we used the fact that $\bcc{[\gamma h]} = \gamma \cdot \bcc{h}$; see~\cite[Prop.~13.23--(i)\&(iv)]{Bauschke17}. 

\subsection{Proof of Corollary~\ref{cor:equivalnece CDP and DP}}
By Proposition~\ref{prop:CDP reform}, we need to show $\bcc{\costu} = \costu$ and $\bcc{[\EE_w J (\cdot+w)]} = \EE_w J (\cdot+w)$ so that
\begin{align*}
\bcc{\costu}(u) + \gamma \cdot \bcc{[\EE_w J (\cdot+w)]} \big( \dyn(x,u) \big) & = \costu(u) + \gamma \cdot [\EE_w J (\cdot+w)] \big( \dyn(x,u) \big) \\
 &= \costu(u) + \gamma \cdot \EE_w J \big( \dyn(x,u) + w\big) \\
 &= \costu(u) + \gamma \cdot \EE_w J \big( \dynw(x,u,w) \big).
\end{align*}
This holds if $\costu$ and $\EE_w J (\cdot+w)$ are proper, closed and convex. 
This is indeed the case since $\setc{X}$ and $\setc{U}$ are compact, and $\costu : \setc{U} \ra \R$ and $J:\setc{X} \ra \R$ are assumed to be convex.

\subsection{Proof of Theorem~\ref{thm:convergence}}
We begin with two preliminary lemmas on the non-expansiveness of conjugate and multilinear interpolation operations within the d-CDP operation~\eqref{eq:d-CDP op}. 

\begin{Lem}[Non-expansiveness of conjugate operator] \label{lem:conj non-exp}
Consider two functions $h_{i}\ (i=1,2)$, with the same nonempty effective domain~$\setc{X}$. 
For any $y \in \dom (\lftc{h_1}) \cap \dom (\lftc{h_2})$, we have 
\begin{align*}
| \lftc{h_1}(y) - \lftc{h_2}(y)| \leq \norm{h_1-h_2}_{\infty}.
\end{align*}
\end{Lem} 

\begin{proof}
For any $y \in \dom (\lftc{h_1}) \cap \dom (\lftc{h_2})$, we have
\begin{align*}
\lftc{h_1}(y) = \max_{x \in \setc{X}} \inner{x}{y} - h_1(x) = \max_{x \in \setc{X}} \inner{x}{y} - h_2(x) + h_2(x) - h_1(x).
\end{align*}
Hence,
\begin{align*}
\lftc{h_2}(y) - \norm{h_1-h_2}_{\infty} \leq \lftc{h_1}(y) \leq \lftc{h_2}(y) + \norm{h_1-h_2}_{\infty},
\end{align*}
that is, 
\begin{align*}
| \lftc{h_1}(y) - \lftc{h_2}(y)| \leq \norm{h_1-h_2}_{\infty}.
\end{align*}
\end{proof}

\begin{Lem}[Non-expansiveness of interpolative LERP operator] \label{lem:lerp non-exp}
Consider two discrete functions $\disc{h}_{i}\ (i=1,2)$ with the same grid-like domain~$\setg{X} \subset \R^n$, and their \emph{interpolative} LERP extensions $\llerp{\disc{h}_{i}}:\co(\setg{X})\ra\R$. 
We have 
\begin{align*}
\norm{\llerp{\disc{h}_{1}}-\llerp{\disc{h}_{2}}}_{\infty} \leq \norm{\disc{h}_1-\disc{h}_2}_{\infty}.
\end{align*}
\end{Lem} 

\begin{proof}
For any $x \in \co(\setg{X})$, we have ($i = 1,2$)
\begin{align*}
\llerp{\disc{h}_{i}} (x) = \sum_{j=1}^{2^n} \alpha^j \ \disc{h}_{i} (x^j),
\end{align*}
where $x^j, \ j=1, \ldots, 2^n$, are the vertices of the hyper-rectangular cell that contains $x$, and $\alpha^j, \ j=1, \ldots, 2^n$, are convex coefficients (i.e., $\alpha^j \in [0,1]$ and $\sum_j \alpha^j = 1$). 
Then  
\begin{align*}
\left|\llerp{\disc{h}_{1}} (x) - \llerp{\disc{h}_{2}} (x)\right| \leq \sum_{j=1}^{2^n} \alpha^j \ \left|\disc{h}_{1} (x^j) - \disc{h}_{2} (x^j)\right| \leq \norm{\disc{h}_1-\disc{h}_2}_{\infty}.
\end{align*}
\end{proof}

With these preliminary results at hand, we can now show that $\dcdpo$ is $\gamma$-contractive. 
Consider two discrete functions $\disc{J}_{i}:\setd{X}\ra \R \ (i=1,2)$. 
For any $x \in \setd{X} \subset \R^n$, we have
\begin{align*}
\left| \dcdpo \disc{J}_1 (x)- \dcdpo \disc{J}_2 (x) \right| &\overset{\eqref{eq:d-CDP op TV}}{=} \left| \llerp{\lftdd{\varphi}_1} \big( \dynx(x) \big) - \llerp{\lftdd{\varphi}_2} \big( \dynx(x) \big) \right| 
\overset{\text{Lem.}~\ref{lem:lerp non-exp}}{\leq} \norm{\lftdd{\varphi}_1 - \lftdd{\varphi}_2}_{\infty} \\ 
&\overset{\text{Def.}}{\leq} \norm{\lftd{\varphi}_1 - \lftd{\varphi}_2}_{\infty} 
\overset{\text{Lem.}~\ref{lem:conj non-exp}}{\leq} \norm{\disc{\varphi}_1 - \disc{\varphi}_2}_{\infty} 
\overset{\eqref{eq:d-CDP op phi}}{\leq} \norm{\lftdd{\varepsilon_1} - \lftdd{\varepsilon_2}}_{\infty} \\
&\overset{\text{Def.}}{\leq}  \norm{\lftd{\varepsilon_1} - \lftd{\varepsilon_2}}_{\infty} 
\overset{\text{Lem.}~\ref{lem:conj non-exp}}{\leq} \norm{\disc{\varepsilon_1} - \disc{\varepsilon_2}}_{\infty} \\
&\overset{\eqref{eq:d-CDP op eps} }{=} \gamma \cdot \norm{\sum_{w \in \setd{W}} p(w) \cdot \left(\lerp{\disc{J}_1} (x+w) - \lerp{\disc{J}_2} (x+w) \right)}_{\infty} \\
&\leq \gamma \cdot \norm{\lerp{\disc{J}_1} - \lerp{\disc{J}_2}}_{\infty} \leq \gamma \cdot \norm{\disc{J_1} - \disc{J_2}}_{\infty}.
\end{align*}
We note that we are using: (i)~Assumption~\ref{As:grids}-\ref{As:grid Z} in the application of Lemma~\ref{lem:lerp non-exp}, (ii)~the fact that $\dom (\lftd{\varphi}_{i}) = \dom (\lftd{\varepsilon_i}) = \R^n$ for $i=1,2$ in the two applications of Lemma~\ref{lem:conj non-exp}, and (iii)~Assumption~\ref{As:extension op}-\ref{As:ext nonexp} in the last inequality.

\subsection{Proof of Theorem~\ref{thm:complexity}}
In what follows, we provide the time complexity of each line of Algorithm~\ref{alg:d-CDP separ cost}. 
In particular, we use the fact that $Y,Z = X$ and $V= U$ by Assumption~\ref{As:grids}-\ref{As:grid complexity}. 
The complexity of construction of $\setg{V}$ in line~\ref{line_alg2:const V} is of $\ord (X+U)$ by Assumption~\ref{As:grids}-\ref{As:grid complexity}. 
The LLT of line~\ref{line_alg2:LLT of Ci} requires $\ord (U+V) = \ord(U)$ operations~\cite[Cor. 5]{Lucet97}. 
The complexity of lines~\ref{line_alg2:const Z} and \ref{line_alg2:const Y} is of $\ord (X+U)$ by Assumption~\ref{As:grids}-\ref{As:grid complexity} on the complexity of construction of $\setg{Z}$ and $\setg{Y}$.
The operation of line~\ref{line_alg2:init 1} also has a complexity of $\ord (X)$, and line~\ref{line_alg2:init 2} requires $\ord(X+U)$ operations. 
This leads to the reported $\ord (X+U)$ time complexity for initialization. 
 
In each iteration, lines \ref{line_alg2:iter 1} requires $\ord(X)$ operations. 
The complexity of line~\ref{line_alg2:LERP of V} is of $\ord(XWE)$ by the assumption on the complexity of the extension operator~$\lerp{[\cdot]}$. 
The LLT of line~\ref{line_alg2:LLT of V} requires $\ord (X+Y) = \ord(X)$ operations~\cite[Cor. 5]{Lucet97}. 
The application of LERP in line~\ref{line_alg2:LERP of Ci} has a complexity of $\ord (\log V)$~\cite[Rem.~2.2]{Kolari21dCDParxiv}. 
Hence, the \texttt{for loop} over~$y \in \setg{Y}$ requires $\ord (Y \log V) = \ord (X \log U) = \wt{\ord}(X)$ operations.
The LLT of line~\ref{line_alg2:LLT of phi} requires $\ord (Z+Y) = \ord(X)$ operations \cite[Cor. 5]{Lucet97}. 
The application of LERP in line~\ref{line_alg2:LERP of phi} has a complexity of $\ord (\log Z)$~\cite[Rem.~2.2]{Kolari21dCDParxiv}. 
Hence, the \texttt{for loop} over~$x \in \setd{X}$ requires $\ord (X \log Z) = \ord (X \log X) = \wt{\ord}(X)$ operations. 
The time complexity of each iteration is then of $\wt{\ord}(XWE)$.

\subsection{Proof of Theorem~\ref{thm:error}}
Note that the ConjVI Algorithm~\ref{alg:d-CDP separ cost} involves consecutive applications of the d-CDP operator~$\dcdpo$~\eqref{eq:d-CDP op}, 
and terminates after a finite number of iterations corresponding to the bound~$e_\mathrm{t}$. 
We begin with bounding the difference between the DP and d-CDP operators. 
We note that this result extends \cite[Thm.~5.3]{Kolari21dCDParxiv} by considering the error of extension operation for computing the expectation w.r.t. to the additive disturbance in~\eqref{eq:d-CDP op eps} and the approximate discrete conjugation of the input cost in~\eqref{eq:d-CDP op phi}.  

\begin{Prop}[Error of d-CDP operation] \label{prop:error d-CDP}
Let $J:\setc{X} \ra \R$ be a Lipschitz continuous, convex function that satisfies the condition of Assumption~\ref{As:extension op}-\ref{As:ext error}. 
Assume $\costu:\setc{U}\ra \R$ is convex. 
Also, let Assumptions~\ref{As:grids}-\ref{As:grid V}\&\ref{As:grid Z} hold.  
Consider the output of the d-CDP operator~$\dcdpo \disc{J}: \setd{X} \ra \R$ and the discretization of the output of the DP operator $\disc{[\dpo J]}: \setd{X} \ra \R$. 
We have
\begin{equation}\label{eq:error d-CDP}
\norm{\dcdpo \disc{J} - \disc{[\dpo J]}}_{\infty} \leq  \gamma \cdot e_\mathrm{e} + e_\mathrm{d}.
\end{equation}
\end{Prop}

\begin{proof}

First note that, by Corollary~\ref{cor:equivalnece CDP and DP}, the DP and CDP operators are equivalent, i.e., $\dpo J = \cdpo J$. 
Hence, it suffices to bound the error of the d-CDP operator~$\dcdpo$ w.r.t. the CDP operator~$\cdpo$. 
We begin with the following preliminary lemma. 

\begin{Lem}\label{lem:eps cont}
The scaled expectation $\epsilon$ in~\eqref{eq:CDP op eps} is Lipschitz continuous and convex with a nonempty, compact effective domain. 
Moreover, $\lip(\epsilon) \leq \gamma\cdot\lip(J)$.
\end{Lem}
\begin{proof}
The convexity follows from the fact that expectation preserves convexity and $\gamma >0$. 
The effective domain of $\epsilon$ is nonempty by the feasibility condition of Assumption~\ref{As:prob data}-\ref{As:constr}, and is compact since $\setc{X}$ is assumed to be compact. 
Finally, the bound on the Lipschitz constant of $\epsilon$ immediately follows from~\eqref{eq:CDP op eps}. 
\end{proof} 

We now provide our step-by-step proof. Consider the function $\epsilon$ in~\eqref{eq:CDP op eps} and its discretization $\disc{\epsilon}:\setd{X} \ra \Ru$. 
Also, consider the discrete function $\disc{\varepsilon}:\setd{X} \ra \Ru$ in~\eqref{eq:d-CDP op eps}. 

\begin{Lem}\label{lem:eps error}
We have $\dom(\disc{\epsilon}) = \dom(\disc{\varepsilon}) \neq \emptyset $. Moreover, 
$ \norm{\disc{\epsilon} - \disc{\varepsilon}}_{\infty} \leq \gamma \cdot e_\mathrm{e} $.
\end{Lem}
\begin{proof}
The first statement follows from the feasibility condition of Assumption~\ref{As:feasible discrete}. 
For the second statement, note that for every $x \in \dom(\disc{\epsilon}) = \dom(\disc{\varepsilon})$, we can use \eqref{eq:CDP op eps} and \eqref{eq:d-CDP op eps} to write
\begin{align*}
\left| \disc{\epsilon}(x) - \disc{\varepsilon}(x) \right| &= \gamma \cdot \left| \sum_{w \in \setd{W}} p(w) \cdot \big( J(x+w) -  \lerp{\disc{J}}(x+w) \big) \right| \\
& \leq \gamma \cdot \sum_{w \in \setd{W}} p(w) \cdot \left| J(x+w) - \lerp{\disc{J}}(x+w) \right| \\
& \leq \gamma \cdot \norm{J - \lerp{\disc{J}}}_{\infty}.
\end{align*}
The result then follows from Assumption~\ref{As:extension op}-\ref{As:ext error} on $J$.
\end{proof} 

Now, consider the function $\phi:\R^n \ra \R$ in~\eqref{eq:CDP op phi} and its discretization $\disc{\phi}:\setg{Y} \ra \R$. 
Also, consider the discrete function $\disc{\varphi}:\setg{Y} \ra \R$ in~\eqref{eq:d-CDP op phi}.

\begin{Lem}\label{lem:phi error}
We have $ \norm{\disc{\phi} - \disc{\varphi}}_{\infty} \leq \gamma \cdot e_\mathrm{e} +  e_\mathrm{u} +  e_\mathrm{v} +  e_\mathrm{x}$, where
\begin{align*}
e_\mathrm{u} &= \left[\norm{B}_2 \cdot \diam{\setg{Y}} + \lip (\costu)  \right] \cdot  \dish (\setc{U}, \setd{U}),  \\
e_\mathrm{v} &= \diam{\setd{U}} \cdot \dish \big(\co( \setg{V}), \setg{V}\big), \\
e_\mathrm{x}  &= \left[ \diam{\setg{Y}} + \gamma \cdot \lip(J) \right] \cdot  \dish (\setc{X}, \setd{X}).
\end{align*}
\end{Lem}
\begin{proof}
Let $y \in \setg{Y}$. According to~\eqref{eq:CDP op phi} and ~\eqref{eq:d-CDP op phi}, we have (note that $\lftdd{\varepsilon}(y) = \lftd{\varepsilon}(y)$)
\begin{align}
\disc{\phi}(y) - \disc{\varphi}(y)  &= \phi(y) - \varphi(y) 
= \lftc{\costu}(-B\tr y) - \llerp{\lftdd{\costu}} (-B\tr y) + \lftc{\epsilon}(y) - \lftd{\varepsilon}(y).\label{eq:Z1}
\end{align}
First, let us use~\cite[Lem.~2.5]{Kolari21dCDParxiv} to write
\begin{align}
0 \leq \lftc{\costu}(-B\tr y) - \lftd{\costu}(-B\tr y) &\leq \big[ \Vert-B\tr y\Vert_2 + \lip (\costu) \big] \cdot \dish (\setc{U},\setd{U}) \nonumber\\
&\leq \left[\norm{B}_2 \cdot \diam{\setg{Y}} + \lip (\costu)  \right] \cdot  \dish (\setc{U}, \setd{U}) = e_\mathrm{u}. \label{eq:Z1-1}
\end{align}
Also, Assumption~\ref{As:grids}-\ref{As:grid V} allows to use~\cite[Cor.~2.7]{Kolari21dCDParxiv} and write
\begin{align}
0 \leq \llerp{\lftdd{\costu}}(-B\tr y) - \lftd{\costu} (-B\tr y) &\leq \diam{\setd{U}} \cdot \dish \big(\co( \setg{V}), \setg{V}\big) = e_\mathrm{v}. \label{eq:Z1-2}
\end{align}
Now, by Lemma~\ref{lem:conj non-exp} (non-expansiveness of conjugation) and Lemma~\ref{lem:eps error}, we have
\begin{align}
\left| \lftd{\epsilon}(y) - \lftd{\varepsilon}(y) \right| \leq  \norm{\disc{\epsilon} - \disc{\varepsilon}}_{\infty} \leq \gamma \cdot e_\mathrm{e}. \label{eq:Z2}
\end{align}
Moreover, we can use~\cite[Lem.~2.5]{Kolari21dCDParxiv} and Lemma~\ref{lem:eps cont} to obtain
\begin{align}
0 \leq \lftc{\epsilon}(y) - \lftd{\epsilon}(y) &\leq \big[ \norm{y}_2 + \lip (\epsilon) \big] \cdot \dish (\setc{X},\setd{X}) \nonumber\\
&\leq \left[ \diam{\setg{Y}} + \gamma \cdot \lip(J) \right] \cdot  \dish (\setc{X}, \setd{X}) = e_\mathrm{x}. \label{eq:Z3}
\end{align}
Combining \eqref{eq:Z1}-\eqref{eq:Z3}, we then have
\begin{align*}
\left| \disc{\phi}(y) - \disc{\varphi}(y) \right| &= \left| \lftc{\costu}(-B\tr y) - \llerp{\lftdd{\costu}} (-B\tr y) + \lftc{\epsilon}(y) - \lftd{\varepsilon}(y) \right| \\
&\leq \left| \lftc{\costu}(-B\tr y) - \lftd{\costu}(-B\tr y) \right| + \left| \lftd{\costu} (-B\tr y) - \llerp{\lftdd{\costu}}(-B\tr y) \right| \\
& \hspace{0.5cm} + \left| \lftc{\epsilon}(y) - \lftd{\epsilon}(y) \right| + \left| \lftd{\epsilon}(y) - \lftd{\varepsilon}(y) \right| \\
& \leq e_\mathrm{u} +  e_\mathrm{v} +  \gamma \cdot e_\mathrm{e} +    e_\mathrm{x}.
\end{align*}
\end{proof}

Next, consider the discrete composite functions $\disc{[\lftc{\phi} \circ \dynx]}: \setd{X} \ra \R $ and $\disc{[\lftd{\varphi} \circ \dynx]}: \setd{X} \ra \R $. 
In particular, notice that $\lftc{\phi} \circ \dynx$ appears in~\eqref{eq:CDP op TV}. 

\begin{Lem}\label{lem:phi error 2}
We have $ \norm{\disc{[\lftc{\phi} \circ \dynx]} - \disc{[\lftd{\varphi} \circ \dynx]}}_{\infty} \leq \gamma \cdot e_\mathrm{e} +  e_\mathrm{u} +  e_\mathrm{v} +  e_\mathrm{x} + e_\mathrm{y}$, where
$$
e_\mathrm{y}  = \big[ \diam{\dynx ( \setd{X} )} + \diam{\setc{X}} + \norm{B}_2 \cdot \diam{\setc{U}} \big] \cdot \max_{x \in \setd{X}} \dist\big(\partial  (\dpo J-\costx) (x),\setg{Y}\big).
$$
\end{Lem}
\begin{proof}
Let $x \in \setd{X}$. 
Also let $\disc{\phi}: \setg{Y} \ra \R$ be the discretization of $\phi: \R^n \ra \R$. 
Since $\phi$ is convex by construction, we can use~\cite[Lem.~2.5]{Kolari21dCDParxiv} to obtain (recall that $\lip ( h; \setc{X})$ denotes the Lipschtiz constant of $h$ restricted to the set $\setc{X} \subset \dom (h)$)
\begin{align}
0 \leq \lftc{\phi}\big(\dynx(x)\big) - \lftd{\phi}\big(\dynx(x)\big) &\leq \min\limits_{y \in \partial  \lftc{\phi}(\dynx(x))} \bigg\{ \big[ \norm{\dynx(x)}_2 + \lip \big( \phi; \{y\} \cup \setg{Y} \big) \big] \cdot \dist(y,\setg{Y}) \bigg\} \label{eq:Z4}
\end{align}
By using~\eqref{eq:CDP op TV} and the equivalence of DP and CDP operators we have $\lftc{\phi} \circ \dynx = \cdpo J-\costx = \dpo J-\costx$. Also, the definition~\eqref{eq:CDP op phi} implies that  
\begin{align*}
\lip ( \phi ) &\leq \lip \big( \lftc{\costu} \circ -B\tr \big) + \lip ( \lftc{\epsilon}) \leq  \norm{B}_2 \cdot \lip( \lftc{\costu} ) + \lip ( \lftc{\epsilon} ) \\
&\leq \norm{B}_2 \cdot \diam{\dom(\costu)} +  \diam{\dom(\epsilon)} \leq \norm{B}_2 \cdot \diam{\setc{U}} + \diam{\setc{X}},
\end{align*}
where for the last inequality we used the fact that $ \dom ( \epsilon ) \subseteq \dom (J) = \setc{X}$. 
Using these results in~\eqref{eq:Z4}, we have 
\begin{align}
0 \leq \lftc{\phi}\big(\dynx(x)\big) - \lftd{\phi}\big(\dynx(x)\big) &\leq \min\limits_{y \in \partial  (\dpo J-\costx) (x)} \bigg\{ \big[ \norm{\dynx(x)}_2 + \diam{\setc{X}} + \norm{B}_2 \diam{\setc{U}} \big] \cdot \dist(y,\setg{Y}) \bigg\} \nonumber \\ 
&\leq \big[ \diam{\dynx ( \setd{X} )} + \diam{\setc{X}} + \norm{B}_2 \cdot \diam{\setc{U}} \big] \cdot \max_{x' \in \setd{X}} \dist\big(\partial  (\dpo J-\costx) (x'),\setg{Y}\big) = e_\mathrm{y}. \label{eq:Z5}
\end{align}
Second, by Lemmas~\ref{lem:conj non-exp} and~\ref{lem:phi error}, we have
\begin{align}
\left| \lftd{\phi} (z) - \lftd{\varphi} (z) \right| \leq \norm{\disc{\phi} - \disc{\varphi}}_{\infty} \leq \gamma \cdot e_\mathrm{e} + e_\mathrm{u} + e_\mathrm{v} +  e_\mathrm{x}, \label{eq:Z6}
\end{align}
for all $z \in \R^n$, including $z = \dynx(x)$. 
Here, we are using the fact that $\dom(\disc{\phi}) = \dom(\disc{\varphi}) = \setg{Y}$ and  $\dom(\lftd{\phi}) = \dom(\lftd{\varphi}) = \R^n$. 
Combining inequalities~\eqref{eq:Z5} and \eqref{eq:Z6}, we obtain
\begin{align*}
\left| \lftc{\phi}\big(\dynx(x)\big) - \lftd{\varphi} \big(\dynx(x)\big) \right| &\leq \left| \lftc{\phi}\big(\dynx(x)\big) - \lftd{\phi}\big(\dynx(x)\big) \right| + \left|\lftd{\phi}\big(\dynx(x)\big) - \lftd{\varphi} \big(\dynx(x)\big) \right| \\
&\leq e_\mathrm{y} + \gamma \cdot e_\mathrm{e} + e_\mathrm{u} +  e_\mathrm{v} +  e_\mathrm{x}.
\end{align*}
This completes the proof.
\end{proof}

We are now left with the final step. 
Consider the output of the d-CDP operator~$\dcdpo \disc{J}: \setd{X} \ra \R$. Also, consider the output of the CDP operator $\cdpo J: \setc{X} \ra \R$ and its discretization $\disc{[\cdpo J]}: \setd{X} \ra \R$. 

\begin{Lem} \label{lem:error d-CDP}
We have
\begin{equation*}
\norm{\dcdpo \disc{J} - \disc{[\cdpo J]}}_{\infty} \leq  \gamma \cdot e_\mathrm{e} + e_\mathrm{u} + e_\mathrm{v} +  e_\mathrm{x} + e_\mathrm{y} + e_\mathrm{z} = \gamma \cdot e_\mathrm{e} + e_\mathrm{d} ,
\end{equation*}
where 
$$
e_\mathrm{z} = \diam{\setg{Y}} \cdot \dish \big(\dynx(\setd{X}), \setg{Z}\big).
$$
\end{Lem}
\begin{proof}
Let $x \in \setd{X}$. According to~\eqref{eq:CDP op TV} and ~\eqref{eq:d-CDP op TV}, we have
\begin{align}
\dcdpo \disc{J} (x) - \disc{[\cdpo J]} (x) = \dcdpo \disc{J} (x) - \cdpo J (x)  = \llerp{\lftdd{\varphi}}\big(\dynx(x)\big) - \lftc{\phi}\big(\dynx(x)\big) \label{eq:Z7}
\end{align}
Now, by Lemma~\ref{lem:phi error 2}, we have
\begin{align}
\left| \lftc{\phi}\big(\dynx(x)\big) - \lftd{\varphi} \big(\dynx(x)\big) \right| \leq \gamma \cdot e_\mathrm{e} + e_\mathrm{u} + e_\mathrm{v} + e_\mathrm{x} + e_\mathrm{y}. \label{eq:Z8}
\end{align}
Moreover, Assumption~\ref{As:grids}-\ref{As:grid Z} allows us to use~\cite[Cor.~2.7]{Kolari21dCDParxiv} and obtain
\begin{align}
0 \leq \llerp{\lftdd{\varphi}}\big(\dynx(x)\big) - \lftd{\varphi} \big(\dynx(x)\big) &\leq \diam{\setg{Y}} \cdot \dish \big(\dynx(\setd{X}), \setg{Z}\big) = e_\mathrm{z}. \label{eq:Z9}
\end{align}
Combining \eqref{eq:Z7}, \eqref{eq:Z8}, and \eqref{eq:Z9}, we then have
\begin{align*}
\left| \dcdpo \disc{J} (x) - \disc{[\cdpo J]} (x) \right| &= \left| \llerp{\lftdd{\varphi}}\big(\dynx(x)\big) - \lftc{\phi}\big(\dynx(x)\big) \right| \\
&\leq \left| \llerp{\lftdd{\varphi}}\big(\dynx(x)\big) - \lftd{\varphi} \big(\dynx(x)\big) \right| + \left| \lftd{\varphi} \big(\dynx(x)\big) - \lftc{\phi}\big(\dynx(x)\big) \right| \\
&\leq \gamma \cdot e_\mathrm{e} + e_\mathrm{u} + e_\mathrm{v} + e_\mathrm{x} + e_\mathrm{y} + e_\mathrm{z} .
\end{align*}
\end{proof}
The inequality~\eqref{eq:error d-CDP} then follows from Lemma~\ref{lem:error d-CDP} by noticing the equivalence of the DP and CDP operators.
\end{proof}

With the preceding result at hand, we can now provide a bound for the difference between the fixed points of the d-CDP and DP operators. 
To this end, let $\disc{\wh{J}\opt} = \dcdpo \disc{\wh{J}\opt}: \setd{X} \ra \R$ be the fixed point of the d-CDP operator. 
Recall that $J\opt = \dpo J\opt : \setc{X} \ra \R$ and $\disc{J\opt}: \setd{X} \ra \R$ are the true optimal value function and its discretization. 

\begin{Lem}[Error of fixed point of d-CDP operator] \label{lem:error d-CDP fp}
We have
\begin{equation*}\label{eq:error d-CDP cor}
\norm{\disc{\wh{J}\opt} - \disc{J\opt}}_{\infty} \leq \frac{\gamma \cdot e_\mathrm{e} +  e_\mathrm{d}}{1-\gamma}.
\end{equation*}
\end{Lem}
\begin{proof}
By Assumptions~\ref{As:grids}-\ref{As:grid Z} and \ref{As:extension op}-\ref{As:ext nonexp}, the operator~$\dcdpo$ is $\gamma$-contractive (Theorem~\ref{thm:convergence}) and hence 
$$\norm{\dcdpo \disc{\wh{J}\opt} - \dcdpo \disc{J\opt}}_{\infty} \leq \gamma \cdot \norm{\disc{\wh{J}\opt} - \disc{J\opt}}_{\infty}.$$ 
Also, notice that Assumptions~\ref{As:prob data} and \ref{As:convex} imply that $J\opt$ is Lipschitz continuous and convex. 
Moreover, $J\opt$ is assumed to satisfy the condition of Assumption~\ref{As:extension op}-\ref{As:ext error}. 
Hence, by Proposition~\ref{prop:error d-CDP}, we have
$$
\norm{ \dcdpo \disc{J\opt} - \disc{[\dpo J\opt]}}_{\infty} \leq \gamma \cdot e_\mathrm{e} + e_\mathrm{d}.
$$
Using these two inequalities, we can then write
\begin{align*}
\norm{\disc{\wh{J}\opt} - \disc{J\opt}}_{\infty} &= \norm{\disc{\wh{J}\opt} - \dcdpo \disc{J\opt} + \dcdpo \disc{J\opt} - \disc{J\opt}}_{\infty}  \\
&\leq \norm{\disc{\wh{J}\opt} - \dcdpo \disc{J\opt}}_{\infty} + \norm{ \dcdpo \disc{J\opt} - \disc{J\opt}}_{\infty}  \\
&= \norm{\dcdpo \disc{\wh{J}\opt} - \dcdpo \disc{J\opt}}_{\infty} + \norm{ \dcdpo \disc{J\opt} - \disc{[\dpo J\opt]}}_{\infty}. \\
&\leq \gamma \cdot \norm{\disc{\wh{J}\opt} - \disc{J\opt}}_{\infty} + \gamma \cdot e_\mathrm{e} + e_\mathrm{d}.
\end{align*}
This completes the proof.
\end{proof}

Finally, we can use the fact that $\dcdpo$ is $\gamma$-cantractive to provide the following bound on the error due to finite termination of the algorithm. 
Recall that $\disc{\wh{J}} : \setd{X} \ra \R$ is the output of Algorithm~\ref{alg:d-CDP separ cost}. 

\begin{Lem}[Error of finite termination]\label{lem:error termination}
We have
\begin{equation*}\label{eq:error termination}
\norm{\disc{\wh{J}} - \disc{\wh{J}\opt}}_{\infty} \leq \frac{\gamma \cdot e_\mathrm{t}}{1-\gamma}.
\end{equation*}
\end{Lem}

\begin{proof}
By Assumptions~\ref{As:grids}-\ref{As:grid Z} and~\ref{As:extension op}-\ref{As:ext nonexp}, the operator~$\dcdpo$ is $\gamma$-contractive (Theorem~\ref{thm:convergence}). 
Let us assume that Algorithm~\ref{alg:d-CDP separ cost} terminates after $k \geq 0$ iterations so that $\disc{\wh{J}} = \disc{J}_{k+1}$ and $\norm{\disc{J}_{k+1} - \disc{J}_k}_{\infty} \leq e_\mathrm{t}$. 
Then,
\begin{align*}
\norm{\disc{\wh{J}} - \disc{\wh{J}\opt}}_{\infty} &= \norm{\disc{J}_{k+1} - \dcdpo \disc{J}_{k+1} + \dcdpo \disc{J}_{k+1} -\disc{\wh{J}\opt}}_{\infty} \\
&\leq \norm{\disc{J}_{k+1} - \dcdpo \disc{J}_{k+1}}_{\infty} + \norm{ \dcdpo \disc{J}_{k+1} -\disc{\wh{J}\opt}}_{\infty} \\
&= \norm{\dcdpo \disc{J}_{k} - \dcdpo \disc{J}_{k+1}}_{\infty} + \norm{ \dcdpo \disc{J}_{k+1} -\dcdpo \disc{\wh{J}\opt}}_{\infty} \\
&\leq \gamma \cdot \norm{\disc{J}_{k} - \disc{J}_{k+1}}_{\infty} + \gamma \cdot \norm{ \disc{J}_{k+1} -\disc{\wh{J}\opt}}_{\infty} \\
&\leq \gamma \cdot e_\mathrm{t} + \gamma \norm{ \disc{\wh{J}} -\disc{\wh{J}\opt}}_{\infty},
\end{align*}
where for the second inequality we used the fact that $\dcdpo$ is a contraction.
\end{proof}

The inequality~\eqref{eq:error} is then derived by combining the results of Lemmas~\ref{lem:error d-CDP fp} and \ref{lem:error termination}.


\bibliographystyle{apalike} 
\begin{small}
\bibliography{ref}
\end{small}

\end{document}